\documentclass[paper=a4, fontsize=11pt]{scrartcl} 

\usepackage[T1]{fontenc} 
\usepackage{fourier} 
\usepackage[english]{babel} 
\usepackage{color}
\definecolor{myblue}{rgb}{.8, .8, 1}
\usepackage{amsmath,amsfonts,amsthm} 
\usepackage{amsmath, calligra, mathrsfs}
\usepackage{amssymb}
\usepackage{empheq}
\usepackage[mathscr]{euscript}
\usepackage{verbatim}
\usepackage{mathtools}
\usepackage[dvipsnames]{xcolor}
\usepackage{graphicx}
\usepackage{caption} 
\usepackage{multicol}
\usepackage{float, graphicx}
\usepackage{mdframed} 
\usepackage{soulutf8}
\usepackage{stmaryrd}
\usepackage{tikz-cd}
\usepackage{hyperref}
\usepackage{lipsum} 

\usepackage{sectsty} 
\allsectionsfont{\centering \normalfont\scshape} 

\usepackage{fancyhdr} 
\usepackage{xurl}
\usepackage{pgfplots}
\usepackage{wasysym}
\usepackage{wrapfig}
\usepackage{array}
\pgfplotsset{compat=1.15}
\usepackage{mathrsfs}
\usetikzlibrary{arrows}
\newcommand*{\sheafhom}{\mathcal{H}\kern -.5pt om}
\makeatletter
\newcommand{\longdash}[1][2em]{%
	\makebox[#1]{$\m@th\smash-\mkern-7mu\cleaders\hbox{$\mkern-2mu\smash-\mkern-2mu$}\hfill\mkern-7mu\smash-$}}
\makeatother
\newcommand{\omitskip}{\kern-\arraycolsep}
\newcommand{\llongdash}[1][2em]{\longdash[#1]\omitskip}
\newcommand{\rlongdash}[1][2em]{\omitskip\longdash[#1]}
\newlength\mytemplen
\newsavebox\mytempbox
\makeatletter
\newcommand\mybluebox{%
	\@ifnextchar[
	{\@mybluebox}%
	{\@mybluebox[0pt]}}

\def\@mybluebox[#1]{%
	\@ifnextchar[
	{\@@mybluebox[#1]}%
	{\@@mybluebox[#1][0pt]}}

\def\@@mybluebox[#1][#2]#3{
	\sbox\mytempbox{#3}%
	\mytemplen\ht\mytempbox
	\advance\mytemplen #1\relax
	\ht\mytempbox\mytemplen
	\mytemplen\dp\mytempbox
	\advance\mytemplen #2\relax
	\dp\mytempbox\mytemplen
	\colorbox{myblue}{\hspace{1em}\usebox{\mytempbox}\hspace{1em}}}

\makeatother
\numberwithin{equation}{section} 
\numberwithin{figure}{section} 
\numberwithin{table}{section} 

\newtheorem{thm}{Theorem}[section]
\newtheorem{cor}[thm]{Corollary}
\newtheorem{prop}[thm]{Proposition}

\newtheorem{lem}[thm]{Lemma}

\newtheorem{quest}[thm]{Question}
\theoremstyle{definition}
\newtheorem{defn}[thm]{Definition}

\newtheorem{exmp}[thm]{Example}

\newtheorem{rem}[thm]{Remark}
\theoremstyle{remark}

\DeclareMathOperator{\Var}{Var}

\DeclareMathOperator{\Sym}{Sym}

\DeclareMathOperator{\Gal}{Gal}

\DeclareMathOperator{\red}{red}

\DeclareMathOperator{\Rep}{Rep}

\DeclareMathOperator{\codim}{codim}

\DeclareMathOperator{\UConf}{UConf}
\DeclareMathOperator{\Conf}{Conf}

\DeclareMathOperator{\Spec}{Spec}

\DeclareMathOperator{\cha}{char}

\DeclareMathOperator{\Aut}{Aut}

\newcommand{\horrule}[1]{\rule{\linewidth}{#1}} 

\title{	
	\normalfont \normalsize 
	\textsc{} \\ [25pt] 
	\horrule{0.5pt} \\[0.4cm] 
	\huge Motivic limits for Fano varieties of $k$-planes 
	\horrule{2pt} \\[0.5cm] 
}

\author{Soohyun Park} 

\date{\normalsize April 24, 2022} 

\begin{document}
	\definecolor{ududff}{rgb}{0.30196078431372547,0.30196078431372547,1.}
	
	\maketitle
	
	\section*{{Abstract}}
	We study the probability that an $(n - m)$-dimensional linear subspace in $\mathbb{P}^n$ or a collection of points spanning such a linear subspace is contained in an $m$-dimensional variety $Y \subset \mathbb{P}^n$. This involves a strategy used by Galkin--Shinder to connect properties of a cubic hypersurface to its Fano variety of lines via cut and paste relations in the Grothendieck ring of varieties. Generalizing this idea to varieties of higher codimension and degree, we can measure growth rates of weighted probabilities of $k$-planes contained in a sequence of varieties with varying initial parameters over a finite field. In the course of doing this, we move an identity motivated by rationality problems involving cubic hypersurfaces to a motivic statistics setting associated with cohomological stability.

	\section{Introduction}

	Given a variety $Y \subset \mathbb{P}^n$ of dimension $m$ and degree $d$, the \emph{Fano variety of $k$-planes} is the subscheme $F_k(Y) \subset \mathbb{G}(k, n)$ parametrizing the set of $k$-planes contained in $Y$. This can be taken to be the Hilbert scheme structure (\cite{AK}, Proposition 6.6 on p. 203 of \cite{EH}) or the reduced structure on it (p. 12 of \cite{GS}).  Since we end up working in the Grothendieck ring of varieties $K_0(\Var_k)$ and $[X] = [X_{\red}]$ in $K_0(\Var_k)$, the nonreduced structure does not play a role our setting and it does not matter which structure we take. In addition, we will take the term ``variety'' to mean an irreducible scheme of finite type. We would like to study the relationships between the following questions: 
	
	\pagebreak 
	
	\begin{quest} \label{mot} ~\\
		\vspace{-3mm}
		\begin{enumerate}
			\item How do properties of $F_k(Y)$ such as arithmetic/geometric invariants vary with initial conditions on $Y$ (e.g. degree, dimension, codimension)?

			\item Given that the Fano variety of $k$-planes has a simple definition as a subvariety of $\mathbb{G}(k, n)$, is there a concrete method (e.g. using a projective geometry construction) other than giving explicit defining equations which give an approach to the first question?
				
			For example, how can we relate the Fano variety of $k$-planes with symmetric products of $Y$ corresponding to unordered $k$-tuples of points in $Y$?

			\item Over \hspace{1mm} $\mathbb{F}_q$, how ``likely'' is a $k$-plane to be contained in $Y$? How does this probability change as we vary $q$? 
			
		\end{enumerate}
	\end{quest}

	The main idea in our approach to Question \ref{mot} (Theorem \ref{avglims}, Corollary \ref{counting}) is to combine two different perspectives on cut--and--paste relations between algebraic varieties. By cut--and--paste relations, we mean the \emph{Grothendieck ring of varieties} $K_0(\Var_K)$ over a field $K$. This is the ring generated by isomorphism classes of algebraic varieties over $K$ quotiented out by relations $[X] = [Z] + [X \setminus Z]$ for closed subvarieties $Z \subset X$ and by $[X \times Y] = [X] [Y]$. \\

	Our starting point is Galkin--Shinder's $Y-F(Y)$ relation, which connects the geometry of a cubic hypersurface with the space of lines on it. 
	
	\begin{thm} [\bfseries{$\mathbf{Y-F(Y)}$ relation}](Galkin--Shinder, Theorem 5.1 on p. 16 of \cite{GS}) \label{yfy} ~\\
		Let $Y \subset \mathbb{P}^{m + 1}$ be a smooth cubic hypersurface of dimension $m$ over an algebraically closed field $K$ of characteristic $0$ and $F(Y) \subset \mathbb{G}(1, m + 1)$ be its (reduced) Fano scheme of lines (p. 12 of \cite{GS}). Then in $K_0(\Var_k)$:
		\begin{equation}
			[Y^{[2]}] = [\mathbb{P}^m][Y] + \mathbb{L}^2 [F(Y)]
			\label{GS}
		\end{equation}

		where $Y^{[2]}$ denotes the Hilbert scheme of two points. 
	\end{thm}

	The relation \eqref{GS} is obtained using a map sending a pair of points $(p, q)$ on $Y$ (considered as an element of $Y^{[2]}$) to the pair $(r, \overline{pq}) \in Y \times \mathbb{G}(1, m + 1)$, where $r$ is the residual point of intersection of $\overline{pq}$ with $Y$ (see Figure \ref{fig:gspic}). This construction gives a correspondence between points of $Y^{[2]}$ spanning a line \emph{not} contained in $Y$ and $(r, \ell) \in \mathbb{G}(1, m + 1)$ such that $r \in \ell$ and $\ell \not\subset Y$. Collecting the non-generic terms coming from lines $\ell \subset Y$ from $Y^{[2]}$ and the incidence correspondence $W = \{ (r, \ell) \in Y \times \mathbb{G}(1, m + 1) : r \in \ell \}$ yields the $Y-F(Y)$ relation. 
	
	\begin{figure}[htb]
		\centering 
		\includegraphics{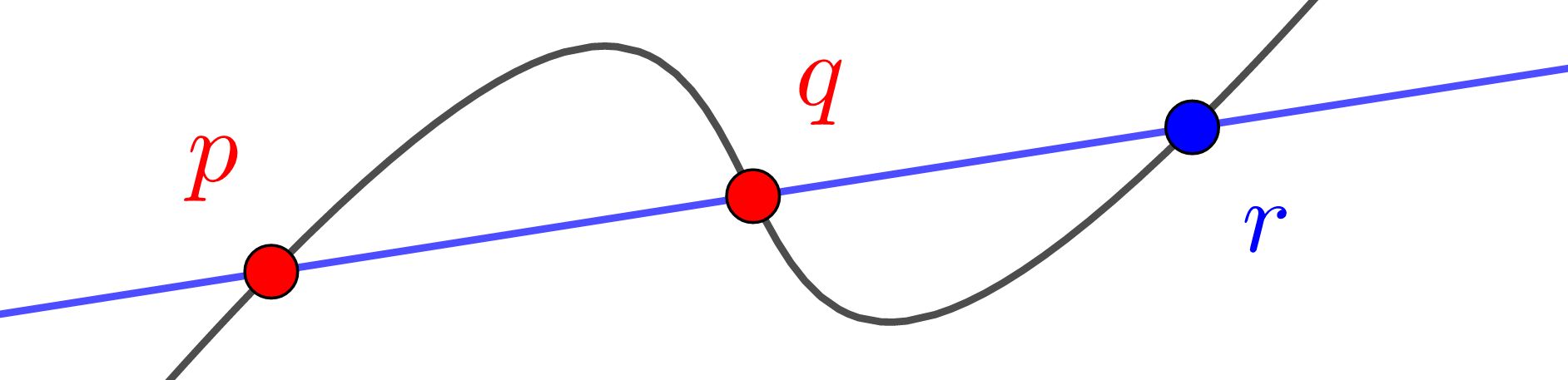} 
		\caption{Sketch of the proof of the $Y-F(Y)$ relation.}
		\label{fig:gspic}
	\end{figure}
	
	As an averaged statement over the space $\mathbb{G}(1, m + 1)$ of lines in $\mathbb{P}^{m + 1}$, the $Y-F(Y)$ relation in Theorem \ref{yfy} can be rewritten as \begin{equation} \label{probyfy}
		\frac{[Y^{[2]}]}{[\mathbb{G}(1, m + 1)]} = \frac{[\mathbb{P}^m][Y]}{[\mathbb{G}(1, m + 1)]} + \frac{([(\mathbb{P}^1)^{(2)}] - [\mathbb{P}^1])[F(Y)]}{[\mathbb{G}(1, m + 1)]} 
	\end{equation}
	
	in $\widehat{\mathcal{K}}$, where $\widehat{\mathcal{K}}$ is a modification of the usual completion with respect to the dimension filtration (Section \ref{extrel}). \\ 
	
	The ``denominator'' $\mathbb{G}(1, m + 1)$ is the total space of lines in $\mathbb{P}^{m + 1}$. Using the argument above, the term  of \ref{probyfy} on the right involving $F(Y)$ gives a weighted probability that a line fails to satisfy the correspondence indicated in Figure \ref{fig:gspic} and the comments above it. \\   
	
	Since the structure of $K_0(\Var_k)$ is compatible with a wide range of invariants including Poincar\'e polynomials and point counts over $\mathbb{F}_q$, the $Y-F(Y)$ relation has many interesting consequences. For example, substituting the Poincar\'e polynomials of $Y$ and the second symmetric product $Y^{(2)}$ into the relation  
	\begin{equation}
		[Y^{(2)}] = (1 + \mathbb{L}^m)[Y] + \mathbb{L}^2 [F(Y)],
		\label{GS2}
	\end{equation}

	which is equivalent to \ref{GS} for $m = 2$,  yields a proof that there are 27 lines on a smooth cubic surface when $\cha k \ne 2$. If $\cha k = 2$, this relation holds modulo universal homeomorphisms since the diagonal morphism $Y \hookrightarrow Y^{(2)}$ is a universal injection (Example 1.1.12 on p. 372 -- 374 of \cite{CNS}, p. 114 of \cite{CNS}). Alternatively, we can localize at radically surjective morphisms as in Section 2.1 of \cite{BH}. \\

	The approach that we take to studying Fano varieties of $k$-planes of a sequence of varieties is to generalize the $Y-F(Y)$ relation \ref{GS}. In Proposition \ref{extend}, we can obtain an analogous relation in $K_0(\Var_k)$ for Fano varieties of $(n - m)$-planes contained in a $m$-dimensional variety $Y \subset \mathbb{P}^n$ of degree $d$. In other words, this is a generalization from lines to $k$-planes of complementary dimension.  Examples of varieties $Y \subset \mathbb{P}^n$ of dimension $m$ containing $(n - m)$-planes are complete intersections of general hypersurfaces of degree $(d_1, \ldots, d_{n - m})$ such that $m \gg \binom{d_i + n - m}{n - m}$ for each $1 \le i \le n - m$ (Theorem 2.4 on p. 4 of \cite{CZ}). \\
	
	The idea is to match up $(k + 1)$-tuples of points of $Y$ lying on a fixed generic $(n - m)$-plane $\Lambda \in \mathbb{G}(n - m, n)$  with the remaining $d - k - 1$ points of $Y \cap \Lambda$ paired with the same $(n - m)$-plane $\Lambda$. Figure \ref{fig:pic0} illustrates this correspondence for a $2$-plane intersecting a variety in $\mathbb{P}^n$ which has codimension $2$ and degree $6$. Incidence correspondences and maps involved in this construction are given in more detail in the proof of Proposition \ref{extend}. \\
		
	Roughly speaking, the $(n - m)$-dimensional linear subspaces contained in a variety parametrize elements of $\mathbb{G}(n - m, n)$ which do \emph{not} give a correspondence between complementary points of intersection. In order to address the complexity of terms involved as the initial parameters are increased, we consider \emph{sequences} of varieties and find an average result. This moves a cut and paste relation coming from a rationality problem into a setting mostly associated with homological stability. We can also give a natural connection between invariants of the variety $Y$ and the space of $(n - m)$-planes contained in it. \\	
		
		\begin{figure}[htb]
			\centering 
			\includegraphics{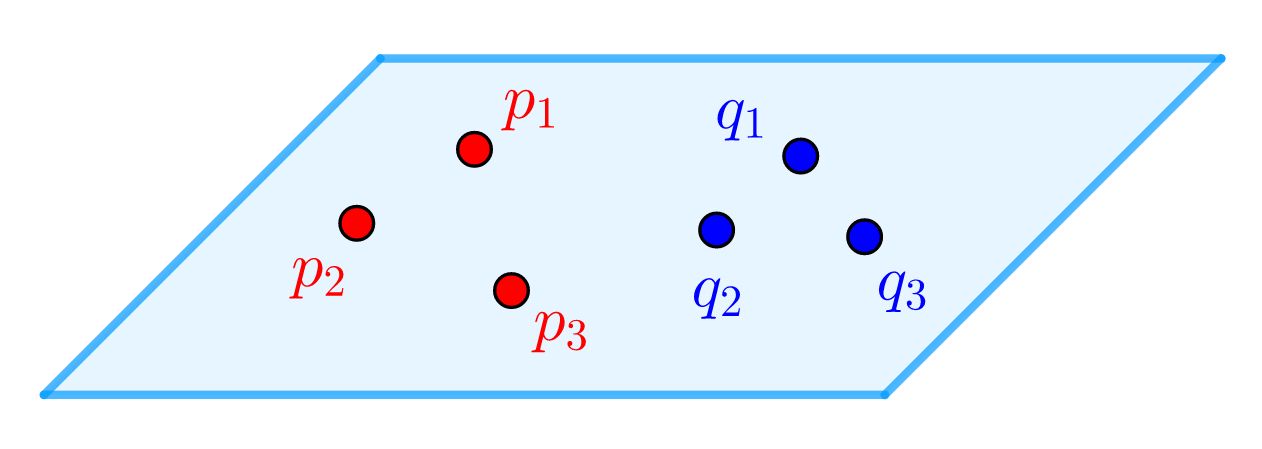}
			\caption{Suppose that $n - m = 2$, $d = 6$, and $k = 2$. The points drawn are the intersection of $Y$ with a generic $(n - m)$-plane. On the left, the $(k + 1)$-tuples of points are drawn in red. The remaining $(d - k - 1)$ points of intersection are drawn in blue.} 
			\label{fig:pic0}
		\end{figure}
		
	As in the original $Y-F(Y)$ relation (Theorem \ref{yfy}), the terms of the higher--dimensional generalization (Proposition \ref{extend}) involving elements of $F_{n - m}(Y) \subset \mathbb{G}(n - m, n)$ parametrizing $(n - m)$-planes contained in $Y$ are contained in the non-generic loci of incidence correspondences (Section \ref{extrel}). However, there are new non-generic terms in this higher dimensional analogue coming from linear dependence between points which did not occur in the $Y-F(Y)$ relation since there are at most two distinct points at once in that case. As a result, the complexity of the terms in $K_0(\Var_K)$ that are used in this extended $Y-F(Y)$ relation increase quickly with the starting parameters to the point of making a simple closed form relation as in the $Y-F(Y)$ relation seems impossible. Our main goal is to extract a meaningful generalization of the $Y-F(Y)$ relation. \\  
	
	In order to do this, we consider ``average'' classes in $K_0(\Var_K)$ over varieties of varying initial parameters (codimension, dimension, degree) and work in a modification $\widehat{\mathcal{K}}$ of the usual completion of $K_0(\Var_k)[\mathbb{L}^{-1}]$ with respect to the dimension filtration (Section \ref{extrel}). More specifically, we show that ``dividing'' by $[\mathbb{G}(n - m, n)]$ gives a sequence of terms in this filtration where the contribution of terms encoding linear dependence approaches $0$. This boils down to dimension computations of the non-generic loci. These terms can be eliminated in the limit since the codimensions in the total spaces increase quickly as the sizes of the inital parameters increase. Before stating the limits which we obtain in this completion, here is some notation.

	\begin{defn} \label{wtavg}
		The \emph{\textbf{approximate weighted average} of linearly independent $u$-tuples of points on $Y = Y_{d, n, m}$} is \[ A_{n, m, u} := \frac{[Y^{(u)}][\mathbb{G}(n - m - u, n - u)]}{[\mathbb{G}(n - m, n)]} \text{ in $\widehat{\mathcal{K}}$ (Section \ref{extrel}),} \]
	\end{defn}
	
	where $Y^{(u)}$ is the $u^{\text{th}}$ symmetric product of $Y$. Note that $d, n, m, u$ are all functions of a single variable $r$ and the limit in the completion is taken as $r \to \infty$ The term $[\mathbb{G}(n - m - u, n - u)]$ parametrizes the set of $(n - m)$-planes that pass through a particular $u$-tuple of linearly independent points of $Y$. If we replace $[Y^{(u)}]$ with the subset $U \subset Y^{(u)}$ consisting of linearly independent $u$-tuples of points, the term $[U][\mathbb{G}(n - m - u, n - u)]$ is the class of \[ \{ ((p_1, \ldots, p_u), \Lambda) \in Y^{(u)} \times \mathbb{G}(n - m - u, n - u) : p_i \in \Lambda \text{ for each $i$ and the $p_i$ are linearly independent} \}. \]

	As stated above, the limits will involve several variables that are all functions of a single variable. The geometric meaning of variables $d, n, m, u$ with $k + 1$ substituted for $u$ is given in Table \ref{tab:variables}.
	
	\begin{defn}\label{seqlim}
		Given a convergent sequence of elements $G_{d, m, n, k}$ of $\widehat{\mathcal{K}}$ with $d = d(r), m = m(r), n = r + m(r)$, and $k = k(r)$ approaching infinity as $r \to \infty$, we will use the following notation for the limit. \[ \widetilde{\lim_{d, m, n, k \to \infty }} G_{d, m, n, k} := \lim_{r \to \infty} G_{d(r), m(r), n(r), k(r)}  \] 
	\end{defn}
	
	The limits will be taken over sequences of varieties of the following form.
	
	\begin{defn} \label{typseq}
		A \textbf{typical sequence} of smooth, closed nondegenerate varieties $Y_{d, m, n} \subset \mathbb{P}^n$ of degree $d$ and dimension $n$ is one where the variables in Table \ref{tab:variables} satisfy the following conditions:

		\begin{enumerate}
			\item For every $r$, $Y_{d(r), m(r), n(r)}$ is contained in a generic hypersurface.
		
			\item There is some $a > 1$ such that $ar \le \dim Y_{d(r), m(r), n(r)}$ for all $r$.

			\item $u$-linearly generic for $u \le d - k - 2$ (Proposition \ref{smallrem}, Proposition \ref{bbound}, Definition \ref{lingen}). This is a genericity condition on hypersurfaces defining certain non-generic linear subspaces of a given dimension in $\mathbb{P}^n$.
		\end{enumerate}
		
	\end{defn}

	Now that the notation is fixed, the limiting extensions of the $Y-F(Y)$ relation in $\widehat{\mathcal{K}}$ can be studied. The cases we will consider are split into the size of the degree relative to the number of sampled points (dots of a single color in Figure \ref{fig:pic0}) and the codimension. In both cases, the $(n - m)$-planes lying inside the given sequence of varieties make up most of the discrepancy between weighted averages of linearly independent $(k + 1)$-tuples and $(d - k - 1)$-tuples under suitable conditions (Definition \ref{wtavg}). \\
	
	We first consider the case where the degree $d$ of the $m$-dimensional variety $Y \subset \mathbb{P}^n$ is not very large with respect to the codimension. In this case, it turns out that $d \le 2(n - m) + 1$ for such varieties. Possible varieties with such a degree have been classified by Ionescu \cite{I1}. The example below (Example \ref{lowexmp}) discusses properties of Fano varieties of $k$-planes of such varieties which can also have arbitrarily large dimension and degree in more detail.

	\begin{exmp}(Low degree examples for Part \ref{endlimit} of Theorem \ref{avglims})  \label{lowexmp} \\
		Examples occuring in the low degree case include scrolls and (hyper)quadric fibrations. Further comments on existence and explicit constructions are given in Example \ref{lowexmpext} at the end of Section \ref{lowlim}.
	\end{exmp}

	Although the discussion in Example \ref{lowexmp} shows that Fano varieties of $k$-planes of many low-degree varieties can be understood using direct computations, we will study them using an averaged $Y-F(Y)$ relation as a way to look at the ``generic'' higher degree case covering \emph{all} other varieties. As in the case of the averaged version \ref{probyfy} of the original $Y-F(Y)$ relation \ref{yfy}, the average is a weighted average of $k$-planes taking tuples of points contained in these planes into account. For this higher degree case, we need some additional notation. Given a variety $Y \subset \mathbb{P}^n$ of dimension $m$ and degree $d$, let   \[ J = \{ ((p_1, \ldots, p_{d - k - 1}), \Lambda) \in Y^{(d - k - 1)} \times \mathbb{G}(n - m, n) : |\Lambda \cap Y| = d \text{ or } \Lambda \subset Y, \text{ $p_i$ distinct, } \dim \overline{p_1, \ldots, p_{d - k - 1}} = n - m \}, \] and let $D \subset (\mathbb{P}^{n - m})^{(d - k - 1)}$ be the set of $(d - k - 1)$-tuples of points spanning a linear subspace of dimension $\le n - m - 1$. The set $J$ parametrizes the set of $(d - k - 1)$-tuples of points in $Y$ which span an $(n - m)$-plane. Note that $[D]$ is a polynomial in $\mathbb{L}$ (Proposition \ref{recursion}). The expressions in the result below give an approximate relation in the high degree case using these objects.
	
	\begin{thm} [\bfseries{Averaged $\mathbf{Y-F(Y)}$ relations}] \label{avglims} ~\\ 
			
		\vspace{-3mm}
		In the expressions below, we consider a sequence of elements of $\mathcal{K}$ which depend on the variables $d, m, n, k$. Each of these variables are functions of a single variable $r$ satisfying certain properties and can be written $d = d(r), m = m(r), n = n(r), k = k(r)$. The limits in $\widehat{\mathcal{K}}$ are taken with respect to $r$ as $r \to \infty$ (Definition \ref{seqlim}). Precise statements on relative dimensions (Definition \ref{reldimdef}) involved in the limit are listed on p. 39 for the low degree case and on p. 37 -- 38 for the high degree case. \\
		
		The limits below are are of sequences indexed by variables which are functions of $r$ that approach infinity as $r \to \infty$. In other words, they can be taken to be limits as $r \to \infty$. \\

		\begin{enumerate}

			\item \textbf{(Low degree case) \label{endlimit}}: Let $Y_{d, n, m}$ be a typical sequence of varieties (Definition \ref{typseq}) of dimension $m$ and degree $d$ such that $\codim_{\mathbb{P}^n} Y_{d, n, m} > 2 \dim Y_{d, n, m} + \Theta(r) $, $\deg Y_{d, n, m} - (k + 1) + \Theta(\sqrt{r}) \le \codim_{\mathbb{P}^n} Y_{d, n, m} - 1$, and $\dim Y_{d, n, m} > 2\deg Y_{d, n, m}$, where $\Theta(f(r))$ denotes being bounded below and above by a constant multiple of $f(r)$ as $r \to \infty$ as usual. If the point sample size $k + 1$ is small with $k \le b \deg Y_{d(r), n(r), m(r)}$ for some $b < 1$, then the limit of this sequence (Definition \ref{seqlim}) is 
			\begin{empheq}[box={\mybluebox[5pt]}]{equation*}
				\hspace{-6mm}\widetilde{\lim_{d, m, n, k \to \infty}} \left( \frac{2[F_{n - m}(Y_{d, n, m})]([(\mathbb{P}^{n - m})^{(k + 1)}] - [(\mathbb{P}^{n - m})^{(d - k - 1)}])}{[\mathbb{G}(n - m, n)]} \right)  - A_{n, m, k + 1} + A_{n, m, d - k - 1} = 0 \text{ in $\widehat{\mathcal{K}}$}. 
			\end{empheq}
		
			This is an expression in the completion. The left hand side gives a sequence of elements in $\mathcal{K}$ which approach $0$ in the completion $\widehat{\mathcal{K}}$. The precise relative dimensions of terms involved are listed in Section \ref{lowlim}.
		
			\item \textbf{(High degree case) \label{highendlim}}: Let $Y_{d, n, m}$ be a sequence of typical varieties (Definition \ref{typseq}) of degree $\deg Y_{d, n, m} - (k + 1) - 1 > n$ and small point sample size $k + 1$ with $k \le br$ for some $b < 1$. Suppose that each $Y_{d, n, m}$ is contained in a complete intersection of $s$ hypersurfaces such that $(n - m)k + k - 1 \ll \sum_{i  = 1}^s \binom{d_i + n - m}{n - m}$. Then the limit of this sequence (Definition \ref{seqlim}) is  
			\begin{empheq}[box={\mybluebox[5pt]}]{align*}
				\widetilde{\lim_{d, m, n, k \to \infty}} \frac{2[F_{n - m}(Y_{d, n, m})]([(\mathbb{P}^{n - m})^{(k + 1)}] - [\UConf_{d - k - 1} \mathbb{P}^{n - m}])}{[\mathbb{G}(n - m, n)]} &- A_{n, m, k + 1} \\
				+ \frac{[J]}{[\mathbb{G}(n - m, n)]} - \frac{2[F_{n - m}(Y_{d, n, m})][D]}{[\mathbb{G}(n - m, n)]} = 0 \text{ in $\widehat{\mathcal{K}}$ if $k - 2 \ll n - m$,} 
			\end{empheq}	 
			
			where $\UConf_e X \subset X^{(e)}$ denotes unordered $e$-tuples of distinct points of $X$. Note that $[D]$ and $\UConf_{d - k - 1} \mathbb{P}^{n - m}$ are polynomials in $\mathbb{L}$ (Proposition \ref{highrecursion}, Lemma \ref{confrec}). \\
			
			As noted in part 1, this sequence of elements in $\mathcal{K}$ which approach $0$ in the completion $\widehat{\mathcal{K}}$. The relative dimensions for this case are listed at the end of Section \ref{highdeg} after Example \ref{uconfexmp}.

		\end{enumerate}

	\end{thm}

	\begin{rem} \label{condsit} ~\\
		\vspace{-5mm}
		\begin{enumerate}
			\item  One consequence is that properties of $J$ compatible with the usual completion of $K_0(\Var_K)$ (e.g. point counts) can be expressed in terms of polynomials in $\mathbb{L}$ and $F_{n - m}(Y)$.
			
			\item The lower bound of $2 \dim Y$ was added in the low degree case to avoid cases where $Y$ is forced to be a complete intersection if Hartshorne's conjecture (p. 1017 of \cite{H3}) holds. This is to ensure that the varieties considered in Part 1 actually have the ``low degree property'' defining that case. Note that this conjecture cannot be strengthened to force a complete intersection outside the range of its original statement (p. 1022 of \cite{H3}). There is also a specific bound in Corollary 3 on p. 588 of \cite{BEL} towards this conjecture and a proof of the conjecture when $n \gg d$ in \cite{ESS}. 
			
			\item There are large parentheses around the first term in the low degree case since it may actually end up vanishing in the completion with the relative dimension (Definition \ref{reldimdef}) approaching $-\infty$. It is stated here since this analogue of the averaged $Y-F(Y)$ is the template for the proof and interpretation of the high degree case. The first term in Part \ref{endlimit} of Theorem \ref{avglims} contains information on $(k + 1)$-tuples or $(d - k - 1)$-tuples contained in an $(n - m)$-plane contained in $Y_{d, n, m}$. On the other hand, the second and third terms parametrize the set of linearly independent $(k + 1)$-tuples and $(d - k - 1)$-tuples paired with an $(n - m)$-plane containing them. The higher degree case also compares maximally linearly independent $(k + 1)$-tuples and $(d - k - 1)$-tuples of points on $Y$ lying on an $(n - m)$-plane.
		\end{enumerate}
	\end{rem}

	When $\deg Y$ is much larger than $\codim_{\mathbb{P}^n} Y$, note that generic $(d - k - 1)$-tuples do \emph{not} lie on an $(n - m)$-plane. In this case, general complete intersections of hypersurfaces of degree $(d_1, \ldots, d_{n - m})$ such that $m \gg \binom{d_i + n - m}{n - m}$ end up being compatible with restrictions on the variables $d, m, n, k$ involved in generalizations of the $Y-F(Y)$ relation (see Example \ref{highexmp} for more details). Substituting these values into the relative dimensions listed in Section \ref{lowlim}, we can see that this limit in Part \ref{endlimit} of Theorem \ref{avglims} can be obtained without assuming $m \gg n - m$ when we take $k = \left\lfloor \frac{m}{2} \right\rfloor$. \\
	
	The proof of the decomposition in the limit in Part \ref{highendlim} of Theorem \ref{avglims} is similar to that of Part \ref{endlimit} of Theorem \ref{avglims}. Note that sufficiently generic complete intersections provide many examples of this higher degree case of Theorem \ref{avglims}. 
	
	\begin{exmp}(High degree examples for Part \ref{highendlim} of Theorem \ref{avglims}: Complete intersections of generic hypersurfaces of large degree) \label{highexmp} \\
		In this case, we an use complete intersections of hypersurfaces which are generic among those of their given degrees. Numerical conditions on possible degrees and their relation to the other variables are explained in more detail in Example \ref{highexmpext} at Section \ref{highlim}.
	\end{exmp}	
	
	Applying the point counting motivic measure and a modified Lang--Weil bound to the limit in Part \ref{highendlim} of Theorem \ref{avglims}, Corollary \ref{counting} gives an upper bound point counts under certain divisibility conditions. While it is not compatible with the dimension filtration (e.g. disjoint union of a curve with a finite collection of points), it is still compatible with the completion (Definition \ref{sep}). This interprets Part \ref{highendlim} of Theorem \ref{avglims} as an answer to Question \ref{mot} on point counts over $\mathbb{F}_q$ as $q$ increases. For example, the probability point count involving tuples of points can be expressed as coefficients of exponential generating functions in terms of point counts of $Y$. \\ 
	
	More specifically, we find a point counting counterpart (Corollary \ref{counting}) to Part \ref{highendlim} of Theorem \ref{avglims} which uses a similar argument. Each term of the latter result comes from a $(k + 1)$-tuple or $(d - k - 1)$-tuple of points on $Y$ paired with an $(n - m)$-plane containing them. For the first and third terms in the result over $\widehat{\mathcal{K}}$, this $(n - m)$-plane is assumed to be contained in $Y$. This is a higher degree analogue of the argument used in Theorem \ref{yfy} for the $Y-F(Y)$ relation. The main difference is that $\# J(\mathbb{F}_q)$ can be expressed in terms of the $\Gal(\overline{\mathbb{F}_q}/\mathbb{F}_q)$-action on $(d - k - 1)$-tuples of points on $Y$.

	\begin{cor} [\bfseries{Averaged Fano $\mathbf{(n - m)}$-plane point count}] 
		\label{counting} ~\\

		Suppose that $d - k - 1 \ge n - m$ and $m, n, d, k$ satisfy the conditions in Part \ref{highendlim} of Theorem \ref{avglims} and $Y \subset \mathbb{P}^n$ is smooth over $\mathbb{F}_q$.  Note that all the variables are functions of $r$ and that $n(r) = m(r) + r$.  As in Theorem \ref{avglims}, the limits are taken with respect to a single variable $r$ of a sequence of $m$-dimensional varieties $Y_{d, m, n} \subset \mathbb{P}^n$ of degree $d$. Each of the variables is a function of $r$.  Given $\Lambda \in \mathbb{G}(n - m, n)$, let
		\begin{align*}
			T_\Lambda = \{\,
			&  N : N \text{ is the number of \text{ }} \\
			& \text{$\Gal(\overline{\mathbb{F}_q}/\mathbb{F}_q)$-orbits  of $(d - k - 1)$-tuples $(p_1, \ldots, p_{d - k - 1}) \in Y^{(d - k - 1)}$ in $Y \cap \Lambda$ for some $\Lambda$} \}. 
		\end{align*}

		Fix a prime power $q$. Let $e = e(r)$ be a positive integer and a function of $r$ such that $e(r) \to \infty$ as $r \to \infty$ and $e > \binom{d}{d - k - 1}$ for all $r$. There is a range of values for $\mathbb{F}_{q^e}$-point counts depending on divisibility properties of \: $\mathbb{F}_q$-irreducible components of slices by linear subspaces of complementary dimension. \\
		
		Given a variety $X$, write $\#_{q, e} X := \# X(\mathbb{F}_{q^e})$. Note that we will consider point counts over varying fields $\mathbb{F}_{q^e}$ with $e \to \infty$ as $r \to \infty$.  This means that 
		
		\vspace{-5mm}
		
		\begin{empheq}[box={\mybluebox[5pt]}]{equation*}
			\hspace{-5mm}\lim_{ r  \to \infty} \frac{\#_{q, e} F_{n - m}(Y_{d, n, m}) (\#_{q, e} \UConf_{d - k - 1}(\mathbb{P}^{n - m}) - \#_{q, e} D - u\#_{q, e} (\mathbb{P}^{n - m})^{(k + 1)} + \#_{q, e} C)}{\#_q \mathbb{G}(n - m, n)} + (1 - u)\alpha + \gamma = 0 
		\end{empheq}
		in the limit for $r = n - m$ and $d = d(r), n = n(r), m = m (r). k = k(r)$, where
		
		\begin{itemize}
			\item $0 \le \alpha \le \binom{d}{d - k - 1}$ with $\alpha = 0$ if $N \nmid e$ for each $N \in T_\Lambda$ from $\Lambda \in \mathbb{G}(n - m, n)$ such that $|Y \cap \Lambda| = d$ and $\alpha = \binom{d}{k + 1}$ if $e$ is divisible by $\binom{d}{d - k - 1}!$ (Proposition \ref{mlw})
			
			\item $u = 1 - \beta + \beta f$, where $\beta = \Theta(q^{e((k + 1)(n - m) + \dim F_{n - m}(Y) - m(n - m + 1))})$ and $f$ is a rational function in ${q^e}$ (mostly) determined by $\frac{[\widetilde{R}]}{[\widetilde{A}]}$, which is a rational function in $\mathbb{L}$ of degree $(k + 1)(n - m) + \dim F_{n - m}(Y) - m(n - m + 1)$
			
			\item $\gamma = \Theta(q^{e(km - (n - m - k + 1) - m(n - m + 1))})$
		\end{itemize} 
		
		are functions that vary with the initial parameters, which depend on $n - m$. Note that the classes $[C]$ and $[D]$ of linearly dependent tuples are polynomials in $\mathbb{L}$ (Proposition \ref{recursion}).
	\end{cor}

	\begin{table}
		\begin{tabular}[ | c | c | c | c | c | c | ]{ | m{2cm} | m{2cm} | m{2cm} | m{2cm} | m{3cm} | m{3cm} | }
			\hline
			\multicolumn{6}{|c|}{\textbf{Parameters used in averaged $\mathbf{Y-F(Y)}$ relations}} \\
			\hline
			\textbf{Variable} & $r$ & $m$ & $d$ & $n$ & $k + 1$ \\ \hline
			\textbf{Definition} & $\codim_{\mathbb{P}^n} Y$ & $\dim Y$ & $\deg Y$ & dimension of projective space in which $Y$ is embedded & number of points on $Y$ on the $(n - m)$-plane for extension of $Y-F(Y)$ construction \\
			\hline
		\end{tabular}
		\caption{\label{tab:variables} Parameters used in Part \ref{endlimit} and Part \ref{highendlim} of Theorem \ref{avglims}.}
	\end{table}

	As a byproduct of the decompositions above, we find a relation between ``how likely'' it is for a $k$-plane to be contained in a variety with the initial parameters such as degree and dimension. This combines some rather different perspectives on applications of the Grothendieck ring of varieties. For example, Galkin--Shinder's work \cite{GS} was motivated by rationality problems involving cubic hypersurfaces while motivic statistics results (\cite{VW}, \cite{BH}) tend to be associated with problems related to cohomological stabilization or point counting. \\
	
	The terms of the direct generalization of the $Y-F(Y)$ relation in $K_0(\Var_k)$ (Proposition \ref{extend} in Section \ref{extrel}) can be split into generic configurations and non-generic configurations which can be arbitrarily complicated as we increase the parameters involved. For this reason, the limits in Part \ref{endlimit} of Theorem \ref{avglims}, Part \ref{highendlim} of Theorem \ref{avglims}, and Corollary \ref{counting} are obtained via upper bounds on the dimensions of the non-generic loci since the completion $\widehat{\mathcal{K}}$ is defined with respect to a dimension filtration (Section \ref{modcpl}). In the dimension counts of Section \ref{dims}, the key idea is to bound the dimensions of the non-generic loci in the extended $Y-F(Y)$ relation Proposition \ref{extend}. These computations are split into the low and high degree cases in Section \ref{lowdeg} and Section \ref{highdeg} respectively. Finally, these dimensions are used to prove Theorem \ref{avglims} by showing that the relative dimensions (Definition \ref{reldimdef}) approach $0$ in Section \ref{lims}. These limiting classes in $\widehat{\mathcal{K}}$ are used to obtain approximate point counts over in $\mathbb{F}_q$ in Corollary \ref{counting}.

	\section*{Acknowledgements}
	
	I am very grateful to my advisor Benson Farb for his guidance and encouragement throughout and providing extensive comments on drafts of this work. Also, I would like to thank the referee for providing thorough comments. 
	
	\section{Dimension computations in $K_0(\Var_k)$ for an extended $Y-F(Y)$ relation}

	In this section, we compute dimensions of varieties in the extended $Y-F(Y)$ relation (Theorem \ref{yfy}) for a smooth nondegenerate irreducible, closed, subvariety $Y \subset \mathbb{P}^n$ of dimension $m$ and degree $d$ that is defined over an algebraically closed field of characteristic $0$. The initial terms to be used in the generalized $Y-F(Y)$ relation are given on p. 11 -- 12 of Section \ref{extrel}. As in the statement of Theorem \ref{avglims}, the cases considered are split into those of low degree (Section  \ref{lowdeg}) and high degree (Section \ref{highdeg}) with respect to the codimension. For the low degree terms, the terms are defined in Proposition  \ref{tuplepar} on p. 18  with dimension counts listed on p. 20. The terms used in the high degree case are defined on p. 34 -- 35. 
	
	\subsection{The extended $Y-F(Y)$-relation} \label{extrel}
	
	Before computing dimensions of the generic and degenerate loci, we first explain components/definitions in an extended $Y-F(Y)$-relation (Proposition \ref{extend}). The idea is to match up linearly independent $(k + 1)$-tuples on the intersection of an $(n - m)$-plane with the residual (linearly independent) $(d - k - 1)$-tuples after removing non-generic loci. These specific constructions assume that $d - k - 1 \le n - m - 1$. The analogous terms for the case $d - k - 1 > n - m - 1$ are listed in Section \ref{highdeg}.  \\

	For each of the total spaces of incidence correspondences ($V$ and $W$ defined below) and non-generic loci inside them, there is a diagram giving the intersection of an $(n - m)$-plane with $Y$ with $q_i$ belonging to a $(d - k - 1)$-tuple and $p_j$ belonging to the residual $(k + 1)$-tuple in the case $d = 6$ and $k = 2$. In the figures below, the points parametrized by the sets defined are filled in. On the other hand, complementary points of intersection of $Y$ with the $(n - m)$-plane are hollow/unfilled (Figure \ref{fig:pic1}, Figure \ref{fig:pic2}, Figure \ref{fig:pic4}). The case involving \emph{both} a $(d - k - 1)$-tuple and its complementary $(k + 1)$-tuple (Figure \ref{fig:pic3}) does not have any hollow/unfilled holes. Finally, the set involving $(n - m)$-planes $\Lambda$ such that $\Lambda \subset Y$ is indicated by having the plane shaded in a new color (Figure \ref{fig:pic4}). \\
	
	We now define the total space $W$ of incidence correspondences of $(d - k - 1)$-tuples in $Y$ paired with an $(n - m)$-plane and stratify the ``non-generic'' loci inside $W$ coming for linear dependence of points or containment of an $(n - m)$-plane in $Y$. 
	\begin{itemize}
		\item $W := \{ ((q_1, \ldots, q_{d - k - 1}), \Lambda) \in Y^{(d - k - 1)} \times \mathbb{G}(n - m, n) : q_i \in \Lambda, \text{ 
			$q_i$ distinct, and either } |Y \cap \Lambda| = d \text{ or }\Lambda \subset Y \}$ (Figure \ref{fig:pic1}) \vspace{-4mm} \begin{figure}[H]
		\centering 
		\includegraphics{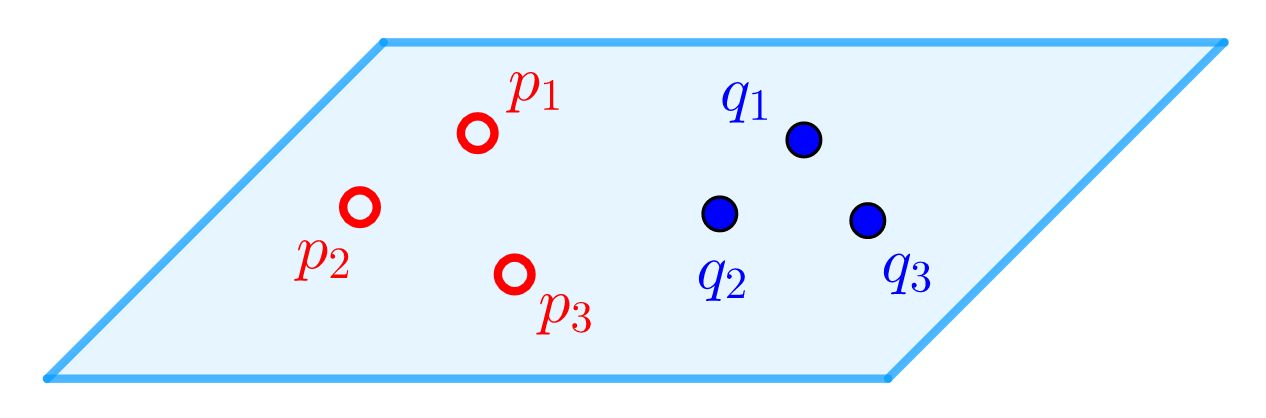}
		\caption{An example configuration. Note that $W$ only considers the triple $q_1, q_2, q_3$ and the $(n - m)$-plane containing them.} 
		\label{fig:pic1}
		\end{figure}

		\item $\widetilde{B} := \widetilde{B}_1 \sqcup \widetilde{B}_2$, where \[ \widetilde{B}_1 := \{ ((q_1, \ldots, q_{d - k - 1}), \Lambda) \in W : q_1, \ldots, q_{d - k - 1} \text{ linearly dependent} \} \text{ (Figure \ref{fig:pic2})} \] \vspace{-10mm} \begin{figure}[H]
		\setlength\belowcaptionskip{-30pt}
		\centering 
		\includegraphics{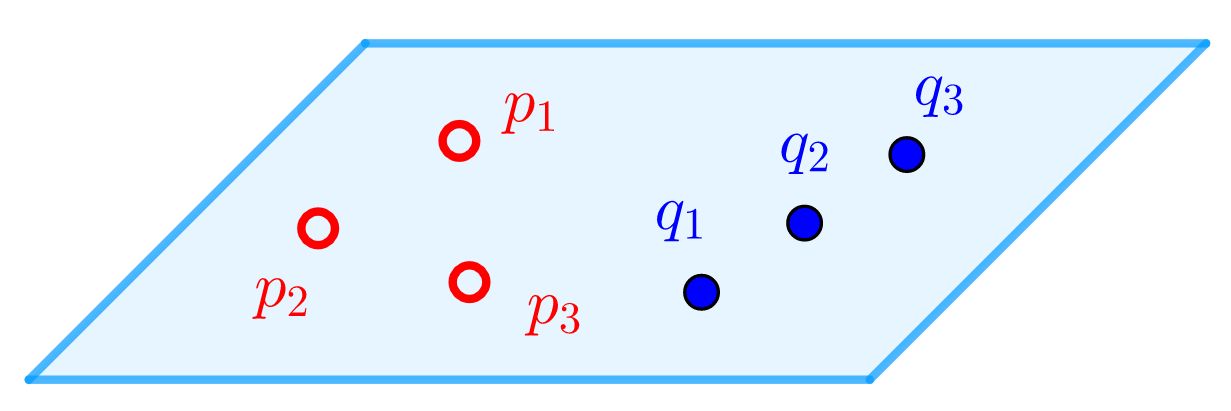} 
		\caption{This is an example of $\widetilde{B}_1$ where $q_1, q_2, q_3$ are linearly dependent. Note that $\widetilde{B}_1$ only imposes conditions on the $q_i$ and not the $p_j$.} 
		\label{fig:pic2}
		\end{figure}
	
		\begin{align*}
			\widetilde{B}_2 &:= \{ ((q_1, \ldots, q_{d - k - 1}), \Lambda) \in W : q_1, \ldots, q_{d - k - 1} \text{ linearly independent but } \\ 
			&(Y \cap \Lambda) \setminus \{ q_1, \ldots, q_{d - k - 1} \} \text{ \emph{not} a linearly independent $(k + 1)$-tuple, } \Lambda \not\subset Y \} \text{ (Figure \ref{fig:pic3})}
		\end{align*} \begin{figure}[H]
		\centering
		\includegraphics{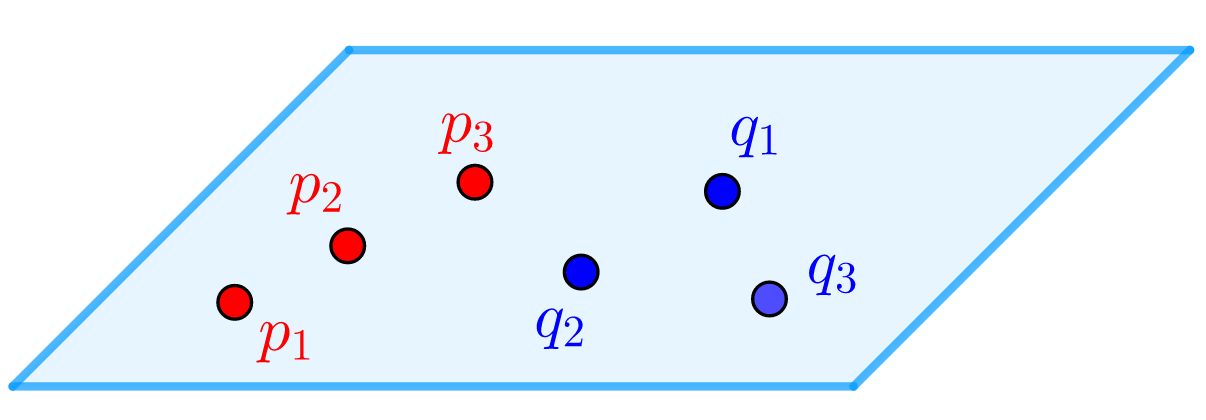} 
		\caption{For $\widetilde{B}_2$, we consider \emph{both} the $(d - k - 1)$-tuple $q_1, \ldots, q_{d - k - 1}$ and the residual points of intersection. While the $q_i$ are linearly independent, the remaining points of $Y \cap \Lambda$ are not.}
		\label{fig:pic3}
		\end{figure}
		
		\pagebreak 
		
		\item $\widetilde{A} := \{ ((q_1, \ldots, q_{d - k - 1}), \Lambda) \in W : q_1, \ldots, q_{d - k - 1} \text{ linearly independent, } \Lambda \subset Y \}$ (Figure \ref{fig:pic4}) \begin{figure}[H]
		\centering
		\includegraphics{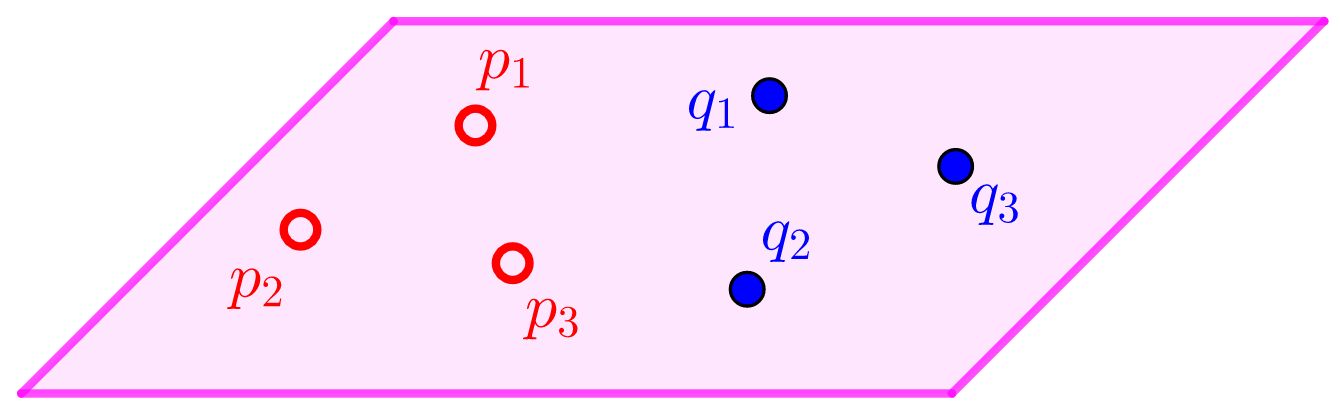}
		\caption{In $\widetilde{A}$, the only linear independence/dependence condition is on the $q_i$ (which we assume to be linearly independent). This is a similar condition to the one defining $\widetilde{B}_1$. Unlike all the terms defined earlier, we assume that the $(n - m)$-plane is contained in $Y$. This is indicated by the change in the color of the $(n - m)$-plane.}
		\label{fig:pic4}
		\end{figure}

		Similar incidence correspondences for $(k + 1)$-tuples of points in $Y$ are defined in the same way except that we switch $k + 1$ and $d - k - 1$ (i.e. switch the $q_i$ with the $p_j$). 
		
		\item $V := \{ ((p_1, \ldots, p_{k + 1}), \Lambda) \in Y^{(k + 1)} \times \mathbb{G}(n - m, n) : p_i \in \Lambda, \text{ $p_i$ distinct, and either } |Y \cap \Lambda| = d \text{ or } \Lambda \subset Y \}$ 
		
		This is an analogue of $W$. 
		
		\item $\widetilde{T} := \widetilde{T}_1 \sqcup \widetilde{T}_2$, where \[ \widetilde{T}_1 := \{ ((p_1, \ldots, p_{k + 1}), \Lambda) \in V : p_1, \ldots, p_{k + 1} \text{ linearly dependent} \} \] and
		\begin{align*}
			\widetilde{T}_2 &:= \{ ((p_1, \ldots, p_{k + 1}), \Lambda) \in V : p_1, \ldots, p_{k + 1} \text{ linearly independent but } \\ 
			&(Y \cap \Lambda) \setminus \{ p_1, \ldots, p_{k + 1} \} \text{ \emph{not} a linearly independent $(d - k - 1)$-tuple, } \Lambda \not\subset Y \}.
		\end{align*}
		
		This is an analogue of $\widetilde{B} = \widetilde{B}_1 \sqcup \widetilde{B}_2$. 
		
		\item $\widetilde{R} := \{ ((p_1, \ldots, p_{k + 1}), \Lambda) \in V : p_1, \ldots, p_{k + 1} \text{ linearly independent, } \Lambda \subset Y \}$ 
		
		This is an analogue of $\widetilde{A}$.

		\item Variable size restrictions:
		\begin{multicols}{3}
			\begin{itemize}			
				\item $d \ge k + 3$ 
				
				\item $d - k - 1 \le n - m - 1$ (only in Section \ref{lowdeg}) \\
				
				\item $k + 1 \le n - m - 1$ 
				
				\item $n - m \le m - 1$ 
				
				\item $d \ge (n - m) + 2$ 
			\end{itemize}
		\end{multicols}

		This last condition on the bottom right is used to ensure that $Y$ is nondegenerate variety that is not a rational normal scroll or Veronese surface (Proposition 0 and Theorem 1 on p. 3 of \cite{EH2}). The remaining conditions come from the incidence correspondences involved in the proof of Proposition \ref{extend}. 
	\end{itemize}
	
	Under the variable restrictions listed above, these incidence correspondences can be used to obtain a higher--dimensional version of the $Y-F(Y)$ relation. Note that the same reasoning implies a higher degree analogue using the analogous objects from Section \ref{highdeg}. In both cases, the idea is to match up maximally linearly independent points on each side. The following proposition relates the various strata of $V$ and $W$.

	\begin{prop}(Extended $Y-F(Y)$-relation) \label{extend} \\
		Suppose that $\overline{k} = k$ and $\cha k = 0$. Then, \[ [W] - [\widetilde{B}] - [\widetilde{A}] = [V] - [\widetilde{R}] - [\widetilde{T}] \text{  in $K_0(\Var_k)$.} \] 
	\end{prop}
	
	\begin{proof}
		
		As with the original $Y-F(Y)$ relation (Theorem 5.1 on p. 16 of \cite{GS}), we show that the residual intersection map induces an equality in $K_0(\Var_k)$ of the non-degenerate loci which are considering. Suppose that $Y \subset \mathbb{P}^n$ is a variety of dimension $m$ of degree $d \le n - m$ over a field $K$ such that $\overline{k} = k$ and $\cha k = 0$. By the definitions on p. 10 -- 11, the term $[W] - [\widetilde{B}] - [\widetilde{A}]$ gives the class of $((p_1, \ldots, p_{d - k - 1}), \Lambda) \in W \subset Y^{(d - k - 1)} \times \mathbb{G}(n - m, n)$ such that $p_1, \ldots, p_{d - k - 1}$ are linearly independent, $\Lambda \not\subset Y$, and $(Y \cap \Lambda) \setminus \{ p_1, \ldots, p_{d - k - 1} \}$ form a linearly independent $(k + 1)$-tuple of points. Let $J \subset W$ be the open subvariety of $((p_1, \ldots, p_{d - k - 1}), \Lambda)$ satisfying these conditions. Similarly, let $K \subset V$ be the subset of pairs $((p_1, \ldots, p_{k + 1}), \Lambda) \subset V \subset Y^{(k + 1)} \times \mathbb{G}(n - m, n)$ such that $p_1, \ldots, p_{k + 1}$ are linearly indendent, $\Lambda \not\subset Y$, and $(Y \cap \Lambda) \setminus \{ p_1, \ldots, p_{k + 1} \}$ form a linearly independent $(d - k - 1)$-tuple of points. In the notation of the definitions listed on p. 11 -- 12, $[K] = [V] - [\widetilde{R}] - [\widetilde{T}]$. \\

		Below, we show that $[J] = [K]$. While the the residual intersection map of the type given in the proof of the $Y-F(Y)$ relation (Theorem 5.1 on p. 16 of \cite{GS}, Example 1.1.12 on p. 372 of \cite{CNS}) gives a bijection between points of $J$ and $K$, it isn't completely obvious why the residual intersection map should give a well-defined morphism/isomorphism. However, we can use projections from an incidence correspondence to show that there are indeed morphisms which induce a bijection of $k$-rational points between $J$ and $K$. The following result implies that this is enough to show equality in $K_0(\Var_k)$: 
		
		\begin{prop}(Proposition 1.4.11 on p. 65 of \cite{CNS}) \label{ratbij} \\
			Let $K$ be a  of characteristic $0$ and let $S = \Spec K$. Let $\overline{K}$ be an algebraically closed extension of $k$. Let $f : Y \longrightarrow X$ be a morphism of $k$-varieties such that the induced map $f(\overline{K}) : Y(\overline{K}) \longrightarrow X(\overline{K})$ is bijective. Then $f$ is a piecewise isomorphism. 
		\end{prop}
		
		Let \[ S = \{ ((p_1, \ldots, p_{d - k - 1}), (q_1, \ldots, q_{k + 1}), \Lambda) : ((p_1, \ldots, p_{d - k - 1}), \Lambda) \in J, ((q_1, \ldots, q_{k + 1}), \Lambda) \in K,  p_i \ne q_j \text{ for all } i, j \}. \] Note that this is a subset of $Y^{(d - k - 1)} \times Y^{(k + 1)} \times \mathbb{G}(n - m, n)$. Consider the projections 
		\begin{center}
			\begin{tikzcd}
				& S \arrow{dr}{\psi} \arrow{dl}[swap]{\varphi} \\
				J \subset W  & & K \subset V 
			\end{tikzcd}
		\end{center}
		given by \[ \varphi : ((p_1, \ldots, p_{d - k - 1}), (q_1, \ldots, q_{k + 1}), \Lambda) \mapsto ((p_1, \ldots, p_{d - k - 1}), \Lambda) \] and \[ \psi : ((p_1, \ldots, p_{d - k - 1}), (q_1, \ldots, q_{k + 1}), \Lambda) \mapsto ((q_1, \ldots, q_{k + 1}), \Lambda). \] Since $\overline{k} = k$, each of the morphisms $\varphi$ and $\psi$ induce bijections of $k$-rational points. Since $\cha k = 0$, Proposition \ref{ratbij} implies that $[S] = [J]$ and $[S] = [K]$. Thus, we have that $[J] = [K]$ and \[  [W] - [\widetilde{B}] - [\widetilde{A}] = [V] - [\widetilde{R}] - [\widetilde{T}] \text{ in $K_0(\Var_k)$} \] as desired. 	
	\end{proof}
	
	\vspace{0.25mm}
	
	\begin{rem} 
		\begin{enumerate}
			\item The methods used in the proof of Proposition \ref{extend} indicate why we made the variable restrictions on p. 12. More specificially, they ensure the existence of linearly independent $k$-tuples of points spanning a linear subspace of dimension $< n - m$ intersecting a variety $Y \subset \mathbb{P}^n$ of degree $d$ and dimension $m$. 
			
			\item Applying the proof of Proposition \ref{extend} to analogues of $S$ with linear dependence among the $p_i$ or $q_j$ implies that the subsets $\widetilde{B}_2$ and $\widetilde{T}_2$ used in the definition of the extended $Y-F(Y)$ relation are constructible since the image of a constructible set under a morphism is constructible. The other sets involved are locally closed via intersections of suitable subsets. 
		\end{enumerate}
	\end{rem}
	
	Apart from working with incidence correspondences matching complementary points of intersection with an $(n - m)$-plane and varying the number of points under consideration, the main difference from the analysis for the original $Y-F(Y)$ relation is that we remove terms involving $(n - m)$-dimensional planes that are tangent to $Y$. We can work out the modified relation in the case of a cubic hypersurface in more detail. Essentially, the idea is to subtract the terms involving tangencies from the initial pairs of points and incidence correspondences. 
	
	\begin{exmp}[\bfseries{Cubic hypersurfaces}] \label{cubic} ~\\
		Let $Y \subset \mathbb{P}^{r + 1}$ be a smooth cubic hypersurface of dimension $r$. In the notation of the extended $Y-F(Y)$ relation, we are setting in $k = 1$, $d = 3$, $m = r$, and $n = r + 1$. Substituting these values into the terms of Proposition \ref{extend} defined on p. 10 -- 12, we have that \[ W = \{ (p, \ell) \in Y \times \mathbb{G}(1, r + 1) : p \in \ell, \text{ and either } |Y \cap \ell| = 3 \text{ or } \ell \subset Y \} \] and \[ V = \{ ((p_1, p_2), \ell) \in Y^{(2)} \times \mathbb{G}(1, r + 1) : p_1 \ne p_2, p_i \in \ell, \text{ and either } |Y \cap \ell| = 3 \text{ or } \ell \subset Y \}. \] This amounts to removing elements such that $\ell$ is tangent to $Y$ and $\ell \not\subset Y$. \\
		
		Next, we can show that $\widetilde{B} = \emptyset$. Note that $\widetilde{B}_1 = \emptyset$ since a single point of $\mathbb{P}^{r + 1}$ cannot be linearly dependent. We also have that $\widetilde{B}_2 = \emptyset$ since the lines involved must be those which are tangent to $Y$ and \emph{not} contained in $Y$ (which we omitted from $W$). Similarly, we have that $\widetilde{T}_2 = \emptyset$ since the only instance where the third point of intersection of $Y$ with a line spanned by distinct points of $Y$ is \emph{not} a ``linearly independent point'' is when one exists. In other words, they must span a line tangent to $Y$. However, we already omitted such lines in the definition of $V$. \\
		
		Thus, we have that \[ W \setminus \widetilde{A} = \{ (p, \ell) \in Y \times \mathbb{G}(1, r + 1) : p \in \ell, |Y \cap \ell| = 3 \} \] and \[ (V \setminus \widetilde{T}) \setminus \widetilde{R} = \{ ((p_1, p_2), \ell) \in Y^{(2)} \times \mathbb{G}(1, r + 1) : p_1 \ne p_2, |Y \cap \overline{p_1, p_2}| = 3 \}. \] Since two distinct points determine a unique line, the second subset can be interpreted as a subset of $Y^{(2)}$. This reduces to the situation of the original $Y-F(Y)$ relation (Theorem 5.1 on p. 16 of \cite{GS}) 
	\end{exmp}
	
	\subsection{A modified completion of $K_0(\Var_k)$} \label{modcpl}
	
	Motivic limits of terms in the extended $Y-F(Y)$ relation can be defined in a modified completion $\widehat{\mathcal{K}}$ of $K_0(\Var_k)$. While the ``averaging'' expressions involving this relation can be defined using a localization by $[\mathbb{G}(n - m, n)]$, this is not necessary and we only need a small modification of the usual completion $\widehat{\mathcal{M}}_k$ to do this. In addition, limits can be defined in a natural way, as is explained in more detail below. \\
	
	Recall that $\widehat{\mathcal{M}}_k := \varprojlim \mathcal{M}_k/F^r \mathcal{M}_k$ is the completion of $\mathcal{M}_k := K_0(\Var_k)[\mathbb{L}^{-1}]$ with respect to the dimension filtration \[ \cdots \subset F^m \mathcal{M}_k \subset F^{m - 1} \mathcal{M}_k \subset \cdots,  \]
	
	where $F^r \mathcal{M}_k$ is the subgroup of $\mathcal{M}_k$ spanned by classes of the form $\frac{[V]}{\mathbb{L}^i}$ with $\dim V - i \le -r$ (p. 8 of \cite{Po}, p. 111 -- 112 of \cite{CNS}). 
	
	\begin{defn}
		Let $\mathcal{K}$ be the extension of scalars of $K_0(\Var_k)$ to $\mathbb{Q}$ made up of $\mathbb{Q}$-linear combinations of classes of varieties over $k$ modulo the same additive relations $[X] = [Y] + [X \setminus Y]$ for closed subvarieties $Y \subset X$ and multiplicative relations $[X] \cdot [Y] = [X \times Y]$. We will write $\mathcal{K}_{\mathbb{C}}$ for the extension of scalars to $\mathbb{C}$. 
	\end{defn}
	
	\begin{defn}
		Let $\widehat{\mathcal{K}}$ be the completion of $\mathcal{K}[\mathbb{L}^{-1}]$ with respect to the dimension filtration $\cdots \subset F^r \mathcal{K}[\mathbb{L}^{-1}] \subset F^{r - 1} \mathcal{K}[\mathbb{L}^{-1}] \subset \cdots$, where $F^r \mathcal{K}[\mathbb{L}^{-1}]$ is the (additive) subgroup of $\mathcal{K}[\mathbb{L}^{-1}]$ spanned by elements of the form $c \frac{[V]}{\mathbb{L}^i}$ with $c \in \mathbb{Q}$ and $\dim V - i \le -r$. Define the completion $\widehat{\mathcal{K}_{\mathbb{C}}}$ of $\mathcal{K}_{\mathbb{C}}[\mathbb{L}^{-1}]$ similarly with the same kind of filtration taking $c \in \mathbb{C}$ instead. 
	\end{defn}

	Using properties of formal power series, we can show that inverting polynomials in $\mathbb{L}$ such as $[\mathbb{G}(n - m, n)]$ is well-defined in the completion $\widehat{\mathcal{K}}$ with respect to the same dimension filtration as the one used for $\widehat{\mathcal{M}_k}$. We may also consider elements in $\widehat{\mathcal{K}_{\mathbb{C}}}$ while attempting to find an explicit expressions for the inverses used. 
	
	\begin{prop} \label{polyinv}
		Given a nonzero polynomial $P \in \mathbb{Q}[T]$, the element $P(\mathbb{L})$ is invertible in $\widehat{\mathcal{K}}$. 
	\end{prop}
	
	\begin{proof}
		Since $\mathbb{L}$ is invertible in $\mathcal{K}$, we can assume without loss of generality that $P(0) \ne 0$. Let $d = \deg P$ and write \[ P(T) = a_d T^d + a_{d - 1} T^{d - 1} + \ldots + a_1 T + a_0. \] If we ``divide'' by $\mathbb{L}^d$ (i.e. multiply by $\mathbb{L}^{-d}$), the resulting expression is a polynomial in $\mathbb{L}^{-1}$ with a nonzero constant term since
		\begin{align*}
			\mathbb{L}^{-d} P(\mathbb{L}) &= \mathbb{L}^{-d}(a_d \mathbb{L}^d + a_{d - 1} \mathbb{L}^{d - 1} + \ldots + a_1 \mathbb{L} + a_0) \\
			&= a_d + a_{d - 1} \mathbb{L}^{-1} + \ldots + a_1 \mathbb{L}^{-(d - 1)} + a_0 \mathbb{L}^{-d}. 
		\end{align*}
		
		Recall that a formal power series with coefficients in a field is invertible as a power series if and only if its constant term is nonzero. Since a polynomial is a (finite) power series and the leading coefficient $a_d \ne 0$, the polynomial $Q(U) = a_d + a_{d - 1} U + \ldots + a_1 U^{d - 1} + a_0 U^d$ has a power series inverse of the form \[ R(U) = c_0 + c_1 U + c_2 U^2 + \ldots \] with $c_i \in \mathbb{Q}$. \\
		
		We claim that $R(\mathbb{L}^{-1})$ gives an expression that is well-defined in $\widehat{\mathcal{K}}$. Since the filtration construction and inverse limit used to define $\widehat{\mathcal{K}}$ is essentially the same as the one used to define $\widehat{\mathcal{M}}_k$, an infinite sum converges in $\widehat{\mathcal{K}}$ if and only if the dimensions of the terms approaches $-\infty$ for the same reasoning as sums in $\mathbb{Q}_p$ (Exercise 2.5 on p. 9 of \cite{Po}). Since $\dim c_m \mathbb{L}^{-m} = -m$ for each $m$, this clearly holds for $R(\mathbb{L}^{-1})$ and this infinite sum is well-defined in $\widehat{\mathcal{K}}$. Thus, the term $\mathbb{L}^{-d} P(\mathbb{L})$ has an inverse in $\widehat{\mathcal{K}}$. Since $\mathbb{L}$ is taken to be invertible, this implies that $P(\mathbb{L})$ itself is invertible in $\widehat{\mathcal{K}}$. 	
	\end{proof}
	
	\begin{rem}
		\begin{enumerate}
			\item For our purposes, it suffices to consider $P \in \mathbb{Z}[T]$ since the denominators in the formal expression for $\frac{[F_k(Y)]}{[\mathbb{G}(k, n)]}$ are polynomials in $\mathbb{L}$ with integer coefficients. 
			
			\item While the proof of Proposition \ref{polyinv} shows that an inverse of $P(\mathbb{L})$ exists in $\widehat{\mathcal{K}}$, it does not say something explicit about what the inverse should look like. In order to obtain some kind of (formal) decomposition, we will work with coefficients in $\mathbb{C}$ using $\widehat{\mathcal{K}_{\mathbb{C}}}$. 
			
			As in the proof of Proposition \ref{polyinv}, we will work with polynomials in $\mathbb{L}^{-1}$. Let $Q(U) = a_0 + a_1 U + \ldots + a_{m - 1} U^{m - 1} + a_m U^m$. Without loss of generality, we can assume that $a_m = 1$. Since we are working over $\mathbb{C}$, this (formally) means that
			\begin{align*}
				\frac{1}{Q(\mathbb{L})} &= \frac{1}{(\mathbb{L} - a_1) \cdots (\mathbb{L} - a_m)} \\
				&= \prod_{r = 1}^m \frac{1}{\mathbb{L} - a_r}
			\end{align*}
			
			for some $a_i \in \mathbb{C}$. \\
			
			For each factor with $a_r \ne 0$, note that
			\begin{align*}
				\frac{1}{\mathbb{L} - c} &= \frac{1}{\mathbb{L}} \cdot \frac{1}{1 - c\mathbb{L}^{-1}} \\
				&= \frac{1}{\mathbb{L}} \cdot \sum_{i = 0}^\infty c^i \mathbb{L}^{-i}
			\end{align*}
			
			by the same reasoning as Exercise 2.7 on p. 9 of \cite{Po}. Substituting this back into our expression for $\frac{1}{Q(\mathbb{L})}$ gives a product of infinite sums with $\mathbb{L}^{-b}$ for some $b$. 
			
		\end{enumerate}
	\end{rem}
	
	\begin{cor} \label{formdef}
		The formal expression for $\frac{[F_k(Y)]}{[\mathbb{G}(k, n)]}$ is well-defined in $\widehat{\mathcal{K}}$. 
	\end{cor}
	
	\begin{proof}
		This follows from applying Proposition \ref{polyinv} to the denominators in the formal expression for $\frac{[F_k(Y)]}{[\mathbb{G}(k, n)]}$, which are polynomials in $\mathbb{L}$. Since $[\mathbb{P}^k] = 1 + \mathbb{L} + \ldots + \mathbb{L}^k$, the fact that symmetric products of sums in $K_0(\Var_k)$ can be expressed as products of symmetric products indexed by partitions (Remark 4.2 on p. 617 of \cite{G1}, p. 6 of \cite{GS}) implies that $[(\mathbb{P}^k)^{(k + 1)}]$ is a polynomial in $\mathbb{L}$. Alternatively, we can use the motivic zeta function $Z_{\mathbb{P}^k}(t) = \frac{1}{(1 - t)(1  - \mathbb{L} t) \cdots (1 - \mathbb{L}^k t)}$ for $\mathbb{P}^k$ (p. 375 of \cite{CNS}) as a generating function for the symmetric product. \\
		
		As for $\mathbb{G}(k, n)$, we use the fact that \[ [\mathbb{G}(k, n)] = [G(k + 1, n + 1)] = \prod_{j = 1}^{k + 1} \frac{\mathbb{L}^{n - k + j} - 1}{\mathbb{L}^j - 1}, \] which follows from a row reduction/Schubert cell argument (Example 2.4.5 on p. 72 -- 73 of \cite{CNS}). 
	\end{proof}

	The grading of a term in $\widehat{\mathcal{K}}$ in the dimension filtration will be called the relative dimension.
	\begin{defn} \label{reldimdef}
		The \emph{relative dimension} of a term $\frac{[P]}{F(\mathbb{L})}$ in $\mathcal{K}[\mathbb{L}^{-1}]$ is $\dim P - \deg F$. This can be extended uniquely to the relative dimension of a term in $\widehat{\mathcal{K}}$ (Remark \ref{dimext}).
	\end{defn}

	\begin{rem} \label{dimext}
		Given a fixed $r$, It is clear how to define the dimension for an element of $\mathcal{K}[\mathbb{L}^{-1}]/F^r \mathcal{K}[\mathbb{L}^{-1}]$. Writing each element of $\widehat{\mathcal{K}}$ as a compatible system of elements of $\mathcal{K}[\mathbb{L}^{-1}]/F^r \mathcal{K}[\mathbb{L}^{-1}]$ for varying $r$, the dimension in $\widehat{\mathcal{K}}$ is defined as the maximum among the dimensions in $\mathcal{K}[\mathbb{L}^{-1}]/F^r \mathcal{K}[\mathbb{L}^{-1}]$ of each nonzero component. Each component is nonzero and of the same dimension if there is \emph{some} nonzero component with positive dimension since the $F^r \mathcal{K}[\mathbb{L}^{-1}]$ keep track of elements with negative dimensions. If all the nonzero components have negative dimensions, all of them actually have the ``maximal'' one $-a$ since any remaining nonzero components only differ by elements of $F^r \mathcal{K}[\mathbb{L}^{-1}]$ for some $r \ge a + 1$. Finally, we take $\dim 0 = 0$ since $\dim c = 0$ for a nonzero constant $c$. 
	\end{rem}
	
	\subsection{Dimension computations} \label{dims}
	
	Now that we have shown that the extended $Y-F(Y)$ relation (Proposition \ref{extend}) holds and defined where motivic limits are taken, we will start to compute the dimensions of terms involved. The computations are split into two subsections according to whether the degree is small or large relative to the codimension. We will first consider the case where $d - k - 1 \le n - m - 1$ in Section \ref{lowdeg}. In Section \ref{highdeg}, similar ideas will be used to obtain dimension counts when $d - k - 1 > n - m - 1$. 
	
	\subsubsection{Low degree nondegenerate varieties ($d - k - 1 \le n - m - 1$)} \label{lowdeg}
	
	There are decompositions of $\widetilde{B}$ and $\widetilde{T}$ that induce a simplification of the identity $[W] - [\widetilde{B}] - [\widetilde{A}] = [V] - [\widetilde{R}] - [\widetilde{T}]$. \\

	\begin{prop}  \label{tuplepar} ~\\
		\vspace{-5mm}
		\begin{enumerate}
			\item The identities \[ [W] - [\widetilde{B}_1] = ([Y^{(d - k - 1)}] - [N])[G(n - m + 1 - (d - k - 1), n + 1 - (d - k - 1))] - [\widetilde{A}] - [P] \] and \[ [V] - [\widetilde{T}] = ([Y^{(k + 1)}] - [M])[G(n - m + 1 - (k + 1), n + 1 - (k + 1))] - [\widetilde{R}] - [Q] \] hold in $K_0(\Var_k)$, where 	
			\begin{itemize}
				\item $N \subset Y^{(d - k - 1)}$ is the set of $(d - k - 1)$-tuples that are linearly \emph{dependent} 
				
				\item $M \subset Y^{(k + 1)}$ is the set of $(k + 1)$-tuples that are linearly \emph{dependent} 
				
				\item $P \subset Y^{(d - k - 1)} \times \mathbb{G}(n - m, n)$ is the set of $((p_1, \ldots, p_{d - k - 1}), \Lambda)$ such that $p_i \in \Lambda$ for each $i$, $\Lambda$ is \emph{not} transversal to $Y$, $\Lambda \not\subset Y$, and $p_1, \ldots, p_{d - k - 1}$ are linearly independent   
				
				\item $Q \subset Y^{(k + 1)} \times \mathbb{G}(n - m, n)$ is the set of $((p_1, \ldots, p_{k + 1}), \Lambda)$ such that $p_i \in \Lambda$ for each $i$, $\Lambda$ is \emph{not} transversal to $Y$, $\Lambda \not\subset Y$, and $p_1, \ldots, p_{k + 1}$ are linearly independent. 
			\end{itemize}
		
		\item Part 1 implies that 
		\begin{align} \label{simrel}
			([Y^{(d - k - 1)}] - [N])[G(n - m + 1 - (d - k - 1), n + 1 - (d - k - 1))] - [P] - [\widetilde{B}_2] \nonumber \\ 
			 - 2[F_{n - m}(Y)]([(\mathbb{P}^{n - m})^{(d - k - 1)}] - [D]) \nonumber \\ 
			= ([Y^{(k + 1)}] - [M])[G(n - m + 1 - (k + 1), n + 1 - (k + 1))] - [Q] - [\widetilde{T}_2] - 2[F_{n - m}(Y)]([(\mathbb{P}^{n - m})^{(k + 1)}] - [C]),  \nonumber \\ 
		\end{align}
		
		where $D \subset (\mathbb{P}^{n - m})^{(d - k - 1)}$ and $C \subset (\mathbb{P}^{n - m})^{(k + 1)}$ are linearly dependent $(d - k - 1)$-tuples and $(k + 1)$-tuples of points in $\mathbb{P}^{n - m}$ respectively.   
		\end{enumerate}
		
	\end{prop}
	
	\begin{proof}
		\begin{enumerate}
			\item The assertions that \[ [W] - [\widetilde{B}_1] + [P] = ([Y^{(d - k - 1)}] - [N])[\mathbb{G}(n - m - (d - k - 1), n + 1 - (d - k - 1))] - [\widetilde{A}] \] and \[ [V] - [\widetilde{T}_1] + [Q] = ([Y^{(k + 1)}] - [M])[\mathbb{G}(n - m - (k + 1), n + 1 - (k + 1))] - [\widetilde{R}] \] in $K_0(\Var_k)$ do not immediately follow from the fibers ``looking the same''. For example, there is no obvious isomorphism from $C \setminus \{ p \}$ to $C \setminus \{ q \}$ for an arbitrary choice of $p, q \in C$ if $C$ is a curve of genus $g \ge 2$ since $\Aut C$ is finite. Also, the map $\mathbb{G}_m \longrightarrow \mathbb{G}_m$ sending $z \mapsto z^d$ seems to indicate that $K_0(\Var_k)$ behaves poorly with respect to covers. However, the map $V \setminus \widetilde{T} \sqcup Q \longrightarrow U$ sending $((p_1, \ldots, p_{k + 1}), \Lambda) \mapsto (p_1, \ldots, p_{k + 1})$ is a piecewise trivial fibration with fiber $\mathbb{G}(n - m - (k + 1), n - (k + 1))$ (Proposition 2.3.4 on p. 70 of \cite{CNS}). \\
			
			Alternatively, we can build a bijection of rational points. Let $U \subset Y^{(k + 1)}$ be the subset consisting of linearly independent $(k + 1)$-tuples of points. Consider the map \[ \varphi : U \times \mathbb{G}(n - m - k - 1, n - k - 1) \longrightarrow (V \setminus \widetilde{T}_1) \sqcup Q \] sending $((p_1, \ldots, p_{k + 1}), \Gamma) \mapsto ((p_1, \ldots, p_{k + 1}l), \langle \overline{p_1, \ldots, p_{k + 1}}, \Gamma \rangle )$, where $\Gamma$ is taken to parametrize $(n - m - k)$-dimensional linear subspaces of the orthogonal complement of $\overline{p_1, \ldots, p_{k + 1}}$ in $\mathbb{A}^{n + 1}$. Note that two elements of $U \times \mathbb{G}(n - m - k - 1, n - k - 1)$ mapping to the same element need to start with the same element of $U$. The second coordinate is the same if and only if the $\Gamma$-coordinates parametrize the same $(n - m - k)$-dimensional linear subspaces of $\mathbb{A}^{n + 1}$. Then, the morphism $\varphi$ induces an injection on $k$-rational points. The morphism $\varphi$ also induces a surjection on $k$-rational points since $\Gamma \cong \overline{p_1, \ldots, p_{k + 1}} \oplus (\Gamma/\overline{p_1, \ldots, p_{k + 1}})$ for any affine linear subspace $\Gamma \supset \overline{p_1, \ldots, p_{k + 1}}$. \\

			Proposition \ref{ratbij} then implies that \[ [V] - [\widetilde{T}_1] + [Q] + [\widetilde{R}] = ([Y^{(k + 1)}] - [M])[G(n - m + 1 - (k + 1), n + 1 - (k + 1))]. \] The same reasoning implies that \[ [W] - [\widetilde{B}_1] + [P] + [\widetilde{A}] = ([Y^{(d - k - 1)}] - [N])[G(n - m + 1 - (d - k - 1), n + 1 - (d - k - 1))]. \] 
			
			\item This follows from rewriting the extended $Y-F(Y)$ relation \[ [W] - [\widetilde{B}] - [\widetilde{A}] = [V] - [\widetilde{R}] - [\widetilde{T}] \] as \[ ([W] - [\widetilde{B}_1] + [P] + [\widetilde{A}]) - [\widetilde{A}] - [P]  - [\widetilde{B}_2] - [\widetilde{A}] = ([V] - [\widetilde{T}_1] + [Q] + [\widetilde{R}]) - [\widetilde{R}] - [Q] - [\widetilde{T}_2] - [\widetilde{R}] \] and making substition from Part 1.
		\end{enumerate}
		
	\end{proof}

	If we work in $\widehat{\mathcal{K}}$ instead, the relation \ref{simrel} in Part 2 of Proposition \ref{tuplepar} can be converted to 
	
	\begin{align*}
		\frac{2[F_{n - m}(Y)]([(\mathbb{P}^{n - m})^{(k + 1)}] - [(\mathbb{P}^{n - m})^{(d - k - 1)}])}{[\mathbb{G}(n - m, n)]} = \frac{([Y^{(k + 1)}] - [M])[G(n - m + 1 - (k + 1), n + 1 - (k + 1))]}{[\mathbb{G}(n - m, n)]} \\ - \frac{([Y^{(d - k - 1)}] - [N])[G(n - m + 1 - (d - k - 1), n + 1 - (d - k - 1))]}{[\mathbb{G}(n - m, n)]} + \frac{[P] - [Q]}{[\mathbb{G}(n - m, n)]} + \frac{[\widetilde{B}_2] - [\widetilde{T}_2]}{[\mathbb{G}(n - m, n)]} \\
		+ \frac{2[F_{n - m}(Y)]([C] - [D])}{[\mathbb{G}(n - m, n)]}. 
	\end{align*}
	
	We will now compute the relative dimensions (i.e. dimensions in $\widehat{\mathcal{K}}$) of the generic and degenerate terms in this identity. The objective of the remainder of this section is to prove the upper bounds for the dimensions of the degeneracy loci listed below. Note that the relative dimensions are listed at the end of this section on p. 22 -- 24.

	\begin{multicols}{2}
		\begin{itemize}
			\item $\dim N \le m(d - k - 2) - 1$ (Proposition \ref{roughbd}) 
			
			\item $\dim M \le mk - 1$ (Proposition \ref{roughbd})

			\item $\dim P \le m(n - m + 1) - m - 1 + (d - k - 1) $ (Proposition \ref{tandim})

			\item $ \dim Q \le m(n - m + 1) - m - 1 + (k + 1)$ (Proposition \ref{tandim})

			\item $\dim \widetilde{B}_2 \le - 2(n - m - (k - 2) - 1)$ as a \emph{relative} dimension in $\widehat{\mathcal{K}}$ (Proposition \ref{bbound}) 
			
			\item $\dim \widetilde{T}_2 \le - 2(n - m - (d - k - 4) - 1)$ as a \emph{relative} dimension in $\widehat{\mathcal{K}}$ (Proposition \ref{multcont}) 
			
			\item $\dim C = (n - m)k + k - 1$ (Lemma \ref{rkdim}) 
			
			\item $\dim D = (n - m)(d - k - 2) + (d - k - 2) - 1$ (Lemma \ref{rkdim}) 
		\end{itemize}
	\end{multicols}

	We first compute $\dim C$ and $\dim D$. The following lemma implies that $\dim D = (n - m)(d - k - 2) + (d - k - 2) - 1$ and $\dim C = (n - m)k + k - 1$. 
	
	\begin{lem} \label{rkdim}
		Let $R_k \subset (\mathbb{P}^n)^{(k + 1)}$ of $(k + 1)$-tuples which form the columns of a $(n + 1) \times (k + 1)$ matrix of rank $\le k$. The dimension of $R_k$ is $nk + k - 1$. 
	\end{lem}
	
	\begin{proof}
		By an incidence correspondence argument, the dimension of the variety $M \subset \mathbb{P}^{(n + 1)(k + 1) - 1}$ of $(n + 1) \times (k + 1)$ matrices of rank $k$ up to scalars is $(nk + n + k) - (n - k + 1) = nk + 2k - 1$ (Proposition 12.2 on p. 151 of \cite{H1}). The quotients by $\mathbb{C}^\times$ and $S_{k + 1}$ indicated in the diagram below imply that $\dim R_k = nk + k - 1$.
		
		\[ \begin{tikzcd}
			(\mathbb{A}^{n + 1})^{k + 1} \arrow{r}{g} \arrow[swap]{d}{f} & (\mathbb{P}^n)^{k + 1} \arrow{d}{\pi} \\
			M \subset \mathbb{P}^{(n + 1)(k + 1) - 1} & (\mathbb{P}^n)^{(k + 1)}
		\end{tikzcd}
		\]		
	\end{proof}
	For our purposes, it suffices to give relatively coarse upper bounds on the dimensions of the degeneracy loci $M$ and $N$. 
	
	\begin{prop} \label{roughbd}
		$\dim M \le mk - 1$ and $\dim N \le m(d - k - 2) - 1$.
	\end{prop}
	
	\begin{proof}
		A generic $(k - 1)$-plane or $(d - k - 3)$-plane spanned by distinct $k$-tuples or $(d - k - 2)$-tuples does \emph{not} intersect an additional point of $Y$. This follows from our assumptions that $k + 1 \le n - m - 1$ and $d - k - 1 \le n - m - 1$ by an argument using the Uniform Position Theorem (p. 370 -- 371 of \cite{EH}). 
	\end{proof}

	The dimension bounds for $P$ and $Q$ follow from the definition of a tangent linear subspace.

	\begin{prop} \label{tandim}
		\[ \dim P \le m(n - m + 1) - m - 1 + (d - k - 1) \] 
		
		with ``relative dimension'' (Definition \ref{reldimdef}) $- m - 1 + (d - k - 1)$
		
		and 
		
		\[ \dim Q \le \dim Q \le m(n - m + 1) - m - 1 + (k + 1) \]
		
		with relative dimension $- m - 1 + (k + 1)$.
	\end{prop}
	
	\begin{proof}
		Without loss of generality, we look at the case of $P$ since the dimension bound for $Q$ has the same proof. Note that $P$ consists of elements of the form $((p_1, \ldots, p_{d - k - 1}), \Lambda)$ with $p_i$ linearly independent, $p \in \Lambda$, $p \in Y$, and $\Lambda$ \emph{not} transversal to $Y$. We can partition the possible $(n - m)$-planes in question into the dimension of the intersection $Y \cap \Lambda$. Let $L_u \subset \mathbb{G}(n - m, n)$ be the space of such $(n - m)$-planes.  \\  
		
		The space $L_u$ is a subset of the space of $(n - m)$-planes $\Lambda \subset \mathbb{P}^n$ such that $\dim (T_q Y \cap T_q \Lambda) \ge u$ for some $q \in Y \cap \Lambda$. In some sense, this measures how far the intersection $q$ is from being transverse. To find the dimension of the latter space, consider the incidence correspondence $J_u = \{ (q, \Lambda) \in Y \times \mathbb{G}(n - m, n) : q \in \Lambda, \dim (T_q Y \cap T_q \Lambda) \ge u \}$ and the projection $\alpha : J_u \longrightarrow \mathbb{G}(n - m, n)$ sending $(q, \Lambda) \mapsto \Lambda$. The definition of $L_u$ implies that $\dim \alpha^{-1}(\Lambda) = u$ for $\Lambda \in L_u$ and $\dim L_u = u + \dim \alpha^{-1}(L_u)$. Using these definitions, our earlier observation can be rewritten as the statement that $\dim L_u \le \dim \alpha(J_u)$. Note that $J_u = \bigcup_{v \ge u} \alpha^{-1}(L_v)$. \\

		To find $\dim J_u$, consider the projection $\beta_u : J_u \longrightarrow Y$ sending $(q, \Lambda) \mapsto q$. Then, we have that $\dim \beta_u^{-1}(q)$ is equal to the dimension of the space of $(n - m)$-planes containing $q$ whose intersection with the tangent plane to $q$ has dimension $\ge u$. To find the dimension of this fiber, we look at a map/projection which parametrizes these $(n - m)$-planes in terms of possible $u$-planes contained in the intersection $Y \cap \Lambda$ (i.e. elements of $\mathbb{G}(u, m)$ which give $u$-dimensional linear subspaces $\Gamma$ of $T_q M \cong \mathbb{P}^m$). For a particular choice of $\Gamma$, the possible $(n - m)$-planes in $\mathbb{P}^n$ containing them is parametrized by elements of $\mathbb{G}(n - m - u, n - u)$. Since $\dim \mathbb{G}(u, m) = (u + 1)(m - u)$ and $\dim \mathbb{G}(n - m - u, n - u) = m(n - m - u + 1)$, we have that $\dim \beta^{-1}(q) = (u + 1)(m - u) + m(n - m - u + 1)$. \\ 
		
		Putting these together, we have that $\dim J_u \le \dim Y + (u + 1)(m - u) + m(n - m - u + 1) = m + (u + 1)(m - u) + m(n - m - u + 1)$. This means that $\dim \alpha(J) = \dim J - u = m + (u + 1)(m - u) + m(n - m - u + 1) - u$. Recall that $\alpha(J)$ is the space of $(n - m)$-planes in $\mathbb{P}^n$ whose intersection with $Y$ has dimension $\ge u$. \\ 
		
		We return to the original incidence correspondence $P$. Consider the projection $\gamma : P \longrightarrow \mathbb{G}(n - m, n)$ sending $((p_1, \ldots, p_{d - k - 1}), \Lambda) \mapsto \Lambda$. The image $\gamma(P)$ can be partitioned into elements of the form $J = J_u$ for some $1 \le u \le \min(m - 1, n - m - 1)$. Note that we will actually take the upper bound is equal to $m - 1$ under the assumptions of Theorem \ref{avglims}. We would like to study $\dim \gamma^{-1}(J_u)$ and see how this varies as we increase $u$ since the $J_u$ partition $\alpha(P)$ and their preimages under $\gamma$ cover $P$. After going from $u$ to $u + 1$, we find that the dimension of the base \emph{decreases} by $2m - u$. In other words, we have that $J(u) - J(u + 1) = 2u + 3$. For the preimages, we find that they \emph{increase} by $d - k - 1$ since the space of possible $p_i$ increases by $1$ from $u$ to $u + 1$ for each $1 \le i \le d - k - 1$. The net change in dimension is then $\dim \gamma^{-1}(\alpha(J_u) - \dim \gamma^{-1}(\alpha(J_{u + 1})) = 2m - u - (d - k - 1)$. If $d - k - 1$ is smaller than $2m$, this means that $\dim J_u$ is a decreasing function in $u$ and the value at $u = 1$ gives the upper bound $\dim P \le m(n - m + 1) - m - 1 + (d - k - 1)$ . Note that the former condition is satisfied under the conditions of Part 1 of Theorem \ref{avglims}. Replacing $d - k - 1$ with $k + 1$, the same reasoning implies that $\dim Q \le m(n - m + 1) - m - 1 + (k + 1)$. Note that the sample size here is very small as given in Part 1 of Theorem \ref{avglims}. \\

	\end{proof}
	
	\vspace{3mm}

	The remaining degeneracy loci whose dimensions we need to compute are $\widetilde{B}_2 \subset W$ and $\widetilde{T}_2 \subset V$. The same method will be used to study each space. \\
	
	We will first study the behavior of $\dim \widetilde{B}_2$. Recall that
	\begin{align*}
		\widetilde{B}_2 &= \{ ((p_1, \ldots, p_{d - k - 1}), \Lambda) \in W : p_1, \ldots, p_{d - k - 1} \text{ linearly independent but } \\ 
		&(Y \cap \Lambda) \setminus \{ p_1, \ldots, p_{d - k - 1} \} \text{ \emph{not} a linearly independent $(k + 1)$-tuple, } \Lambda \not\subset Y \}.
	\end{align*}
	
	Before making dimension computations, here is a small observation which shows that the  result we will use applies to \emph{any} $(k + 1)$-tuple spanning a linear subspace of dimension $\le k - 1$. Afterwards, we state the definition of a term used in the result. 
	
	\begin{lem} \label{disj}
		Any $\mu$-plane $\Gamma \subset \mathbb{P}^n$ that intersects a variety $X \subset \mathbb{P}^m$ at a finite number of points contains a $(\mu - 1)$-plane disjoint from $X$. 
	\end{lem}
	
	\begin{proof}
		To obtain such an $(\mu - 1)$-plane, we can intersect $\Gamma$ with a hyperplane which does \emph{not} contain any points of $\Gamma \cap X$. For example, we can use the hyperplane $x_i = c$ for some $c$ which is not the $i^{\text{th}}$ coordinate of any points of $\Gamma \cap X$. 
	\end{proof}
	
	\begin{defn}(Ran, p. 716 of \cite{R}) \label{projdef} 
		\begin{enumerate} 
			\item Given a subvariety $X \subset \mathbb{P}^m$ and a linear $\lambda$-plane $\Lambda$ disjoint from $X$, denote by $X_k^\Lambda \subset \mathbb{P}^{m - \lambda - 1}$ the locus of fibers of length $k$ or more of the projection $\pi_\Lambda : X \longrightarrow \mathbb{P}^{m - \lambda - 1}$. Thus, $X_k^\Lambda$ is the locus of $(\lambda + 1)$-planes containing $\Lambda$ which meet $X$ in a scheme of length $\ge k$. 
			
			\item The analogous projection and locus of fibers for \emph{generic} $\Lambda$ is denoted $X_k^\lambda \subset \mathbb{P}^{m - \lambda - 1}$. In practice, we state that some property holds for $X_k^\lambda$ when it holds for $X_k^\Lambda$ given a generic choice of $\Lambda \in \mathbb{G}(\lambda, m)$. The meaning of ``generic'' is further explained below in Remark \ref{topcom}.
		\end{enumerate}
	\end{defn}
	
	\pagebreak 
	
	\begin{rem} \label{topcom}
		\item By ``generic'' choice of $\Lambda \in \mathbb{G}(\lambda, m)$, we mean the complement of a nowhere dense analytic subset (p. 699 of \cite{R}). For example, let $H \subset \mathbb{G}(n - m, n)$ be the set of $n - m$-planes which intersect the $m$-dimensional variety $Y \subset \mathbb{P}^n$ at $d := \deg X$ distinct points. This is also an open subset in the Zariski open topology (Corollaire 2.3 on p. 259 of arXiv link and p. 318 in \cite{GR}). Since the intersection of a dense subset with an open subset is dense in the open subset, the intersection of $H$ with the generic locus in Definition \ref{projdef} is dense in $H$. 
	\end{rem}
	
	By Lemma \ref{disj}, the locus of \emph{all} $(\lambda + 1)$-planes which meet $X$ in a scheme of length $\ge k$ is a union of subsets of the form $X_k^\Lambda$ for some $\lambda$-plane $\Lambda$ disjoint from $X$. Here is the main result which we use to prove our claim. 
	
	\begin{thm}(Ran, Theorem 5.1 on p. 716 of \cite{R}) \label{seclen} \\
		Let $X \subset \mathbb{P}^m$ be an irreducible closed subvariety of codimension $c > \lambda \ge 0$. Then $X_k^\lambda$ is smooth of codimension $k(c - \lambda - 1)$ in $\mathbb{P}^{m - \lambda - 1}$, in a neighborhood of any point image of a fiber of length exactly $k$ that is disjoint from the singular locus of $X$ and has embedding dimension $2$ or less. 
	\end{thm}

	\begin{rem} \label{condns}
		\begin{enumerate}
			\item Subsets that have codimension strictly larger than the dimension of the ambient space are taken to be empty (p. 699 of \cite{R}). 
			
			\item In the definition of $V$ and $W$ from the extended $Y-F(Y)$ relation (Proposition \ref{extend}), we assumed that $|Y \cap \Lambda| = d$ for $(n - m)$-planes . Since the curvilinear subscheme of a Hilbert scheme of $r$ points on a smooth projective variety is formed by the closure unordered tuples of $r$ distinct points, we will study curvilinear schemes in our setting. Any $\lambda$-plane ($\lambda \le n - m - 1$) contained in these $(n - m)$ planes satisfies the embedding dimension condition of Theorem \ref{seclen} since curvilinear schemes have local embedding dimension $\le 1$ (p. 703 of \cite{R}). 
			
			\item While $n$ is not explicitly defined in the statement of Corollary 5.6 on p. 717 of \cite{R} (or anywhere in Section 5 of \cite{R}), it is indicated that Theorem \ref{seclen} is a partial extension of Theorem 4.1 on p. 713 of \cite{R}, which makes use of this notation. 
		\end{enumerate}
	\end{rem}
	
	Given $\lambda \le k - 2$, we can use this to compute the dimension of the space of non-tangent $(\lambda + 1)$-planes intersecting $Y$ at a linearly dependent $(k + 1)$-tuple of points. Similarly, the same method can be used for $\lambda \le d - k - 2$ and $(d - k - 1)$-tuples of points. \\
	
	\begin{prop} \label{lam1}
		Suppose that $Y \subset \mathbb{P}^n$ is a smooth closed irreducible variety of dimension $m$ and degree $d$. 
		
		\begin{enumerate}
			\item Given $\lambda \le k - 2$, the space of $(\lambda + 1)$-planes in $\mathbb{P}^n$ intersecting $Y$ at $\ge k + 1$ points and contain \emph{some} generic $\lambda$-plane (in the sense of Theorem \ref{seclen}) which are \emph{not} tangent to $Y$ or contained in $Y$ has dimension $\dim \mathbb{G}(\lambda, n) + \dim Y_{k + 1}^\lambda - (\lambda + 1)$. 
			
			\item Given $\lambda \le d - k - 2$, the space of $(\lambda + 1)$-planes in $\mathbb{P}^n$ intersecting $Y$ at $\ge d - k - 1$ points and contain \emph{some} generic $\lambda$-plane which are \emph{not} tangent to $Y$ or contained in $Y$ has dimension $\dim \mathbb{G}(\lambda, n) + \dim Y_{d - k - 1}^\lambda - (\lambda + 1)$. 
		\end{enumerate}
	\end{prop}

	\begin{proof}
		\begin{enumerate}
			\item Let \[ \mathcal{A} = \{ (\Lambda, \Gamma) \in \mathbb{G}(\lambda, n) \times \mathbb{G}(\lambda + 1, n) : \Lambda \subset \Gamma,  \Lambda \text{ generic and not tangent to $Y$}, |Y \cap \Lambda| \ge k + 1, \Lambda \not\subset Y \}. \]
			
			Consider the projections
			
			\begin{center}
				\begin{tikzcd}
					& \mathcal{A} \arrow{dr}{\psi} \arrow{dl}[swap]{\varphi} \\
					\mathbb{G}(\lambda, n)  & & \mathbb{G}(\lambda + 1, n). 
				\end{tikzcd}
			\end{center}
			
			In this diagram, the space of $(\lambda + 1)$-planes intersecting $Y$ at $\ge k + 1$ and contain \emph{some} generic $\lambda$-plane which are \emph{not} tangent to $Y$ or contained in $Y$ is given by $\psi(\mathcal{A})$. Thus, it suffices to compute $\dim \psi(\mathcal{A})$. \\
			
			Let $U \subset \mathbb{G}(\lambda, n)$ be the space of generic $\lambda$-planes not tangent to $Y$. Then, the definition of $ \mathcal{A} $ implies that $U = \varphi(\mathcal{A})$. For each $\Lambda \in U$, we have that $\dim \varphi^{-1}(\Lambda) = \dim Y_{k + 1}^\lambda$. On the other hand, we have that $\psi^{-1}(\Gamma) \subset \mathbb{G}(\lambda, \lambda + 1) \cong \mathbb{P}^{\lambda + 1}$ is a nonempty open subset for each $\Gamma \in \psi(\mathcal{A} )$. Since $\mathbb{P}^{\lambda + 1}$ is irreducible, this implies that $\psi^{-1}(\Gamma)$ forms a dense open subset. Thus, we have that $\dim \psi^{-1}(\Gamma) = \lambda + 1$ for each $\Gamma \in \psi( \mathcal{A} )$. Putting these together, Corollary 11.13 on p. 139 of \cite{H1} implies that 
			\begin{align*}
				\dim \psi( \mathcal{A} ) &= \dim \mathcal{A} - (\lambda + 1) \\
				&= \dim U + \dim Y_{k + 1}^\lambda - (\lambda + 1) \\
				&= \dim \mathbb{G}(\lambda, n) + \dim Y_{k + 1}^\lambda - (\lambda + 1). 
			\end{align*}
		
		\item This uses the same steps as part 1 except that $k$ is replaced with $d - k - 2$. 
		\end{enumerate}
	\end{proof}
	
	Under appropriate conditions, we can omit $(n - m)$-planes that do \emph{not} contain any generic $(n - m - 1)$-planes in $\mathbb{P}^n$.

	\begin{prop} \label{smallrem}
		Let $Z \subset \mathbb{G}(n - m, n)$ be the complement of the locus of generic $(n - m - 1)$-planes (Theorem \ref{seclen}) and $\mathbb{G}(n - m - 1, n) \hookrightarrow \mathbb{P}^N$ be the Pl\"ucker embedding. If $Z$ is contained in some hypersurface generic in its degree, then the set of $(n - m)$-planes in $\mathbb{P}^n$ whose $(n - m - 1)$-subplanes are all contained in $Z$ form a finite subset which is empty if the degree is $\ge 2$. Note that \emph{some} condition is necessary in order to have such a codimension.  
	\end{prop}
	
	\begin{proof}
		
		Each polynomial $F$ in the ideal defining $Z \subset \mathbb{G}(n - m - 1, n)$ can be considered as a polynomial in affine (Pl\"ucker) coordinates $(a_{i, j})_{\substack{1 \le i \le n - m \\ 1 \le j \le m + 1 }}$ with $F$ modified depending on the specific chart $\begin{pmatrix} I_{n - m} & A \end{pmatrix}$ by precomposing with right multiplication by some element of $GL_{n + 1}$. Recall that the standard affine chart of $G(r, n)$ corresponds to $r$-dimensional linear subspaces which do \emph{not} intersect a specific $(n - r)$-plane nontrivially and transition maps are given by $GL_r$-actions. In the statement above, the $(n - m)$-planes $\Gamma \in \mathbb{G}(n - m, n)$ are exactly those such that $\Gamma \cap H \in Z$ for \emph{all} hyperplanes $H \subset \mathbb{P}^n$. We can relate this back to the usual affine chart on $\mathbb{G}(n - m, n)$. \\ 
		
		 On a standard chart for $\mathbb{G}(n - m, n) = G(n - m + 1, n + 1)$, we can represent $\Gamma \in \mathbb{G}(n - m, n)$ as a matrix of the form $\begin{pmatrix} I_{n - m + 1} & B \end{pmatrix}$, where $B = (b_{r, s})$ is an $(n - m + 1) \times m$ matrix corresponding to an element of $\mathbb{A}^{m(n - m + 1)}$. Recall that we wanted to have $\Gamma \cap H \in Z$ for each hyperplane $H \subset \mathbb{P}^n$. Note that \emph{rows} of $\begin{pmatrix} I_{n - m + 1} & B \end{pmatrix}$ represent elements of $\mathbb{P}^n$. Let $\beta_i = (b_{i, 1}, \ldots, b_{i, m})$ be the $i^{\text{th}}$ row of $B$. Given $\Gamma \in \mathbb{G}(n - m, n)$ and its chart representation $\begin{pmatrix} I_{n - m + 1} & B \end{pmatrix}$, the $(n - m - 1)$-dimensional subspaces of $\Gamma$ correspond to elements of the form

		\begin{align*}
			\begin{pmatrix} 1 & 0 & \cdots & 0 & \alpha_1 \\ 0 & 1 & \cdots & 0 & \alpha_2 \\ \vdots & \vdots & \cdots & \vdots & \vdots \\ 0 & 0 & \cdots & 1 & \alpha_{n - m} \end{pmatrix} \cdot g \cdot \begin{pmatrix} 1 & 0 & \cdots & 0 & 0 & \llongdash & \beta_1 & \rlongdash \\ \vdots & \vdots & \cdots & \vdots & \vdots & \vdots & \vdots & \vdots \\ 0 & 0 & \cdots & 0 &  1 & \llongdash & \beta_{n - m + 1} & \rlongdash \end{pmatrix}  
		\end{align*}

		for some $g \in GL_{n - m + 1}$. The rows of the product of the first matrix with $g$ give bases of $(n - m - 1)$-dimensional subspaces of $\Gamma$ with respect to the basis given by the rows of the last matrix. The first matrix can be rewritten as the $(n - m) \times (n - m + 1)$ matrix $\begin{pmatrix} I_{n - m} & \alpha \end{pmatrix}$ for some $\alpha \in \mathbb{A}^{n - m}$ and the second one is the $(n - m + 1) \times n$ matrix $\begin{pmatrix} I_{n - m + 1} & B \end{pmatrix}$ representing $\Gamma$ with $\beta_i = (b_{i, 1}, \ldots, b_{i, m})$ the $i^{\text{th}}$ row of $B$. The term $\begin{pmatrix} I_{n - m} & \alpha \end{pmatrix}$ comes from considering representations of $(n - m - 1)$-dimensional subspaces of $\mathbb{P}^{n - m}$ and $g$ is gives a change of basis/change of coordinates which moves between charts in the affine covers of $\mathbb{G}(n - m, n)$ and $\mathbb{G}(n - m - 1, n)$ which	we are using here. We will first consider the case $g = I_{n - m + 1}$ and reduce the general case to this afterwards. \\
		 		
		In these coordinates, the $(n - m - 1)$-dimensional subspaces of the $(n - m)$-dimensional linear subspace $\Gamma$ represented by $\begin{pmatrix} I_{n - m + 1} & B \end{pmatrix}$ satisfy $(F = 0)$ (under the appropriate chart/multiplication by an element of $GL_{n - m}$) if and only if $F = 0$ on the $(n - m) \times (m + 1)$ submatrix  
		
		\begin{align*}
			\begin{pmatrix} 1 & 0 & \cdots & 0 & \alpha_1 \\ 0 & 1 & \cdots & 0 & \alpha_2 \\ \vdots & \vdots & \cdots & \vdots & \vdots \\ 0 & 0 & \cdots & 1 & \alpha_{n - m} \end{pmatrix} \begin{pmatrix} 0 & \llongdash & \beta_1 & \rlongdash \\ \vdots & \vdots & \vdots & \vdots \\ 1 & \llongdash & \beta_{n - m + 1} & \rlongdash \end{pmatrix} =  \begin{pmatrix} \alpha_1 & \llongdash & \beta_1 + \alpha_1 \beta_{n - m + 1} & \rlongdash \\ \vdots & \vdots & \vdots & \vdots \\ \alpha_{n - m} & \llongdash & \beta_{n - m} + \alpha_{n - m} \beta_{n - m + 1} & \rlongdash \end{pmatrix}\\ 
			= \begin{pmatrix} \alpha_1 & \llongdash & \beta_1 + \alpha_1 \beta_{n - m + 1} & \rlongdash \\ \vdots & \vdots & \vdots & \vdots \\ \alpha_{n - m - 1} & \llongdash & \beta_{n - m - 1} + \alpha_{n - m - 1} \beta_{n - m + 1} & \rlongdash \\ 0 & \llongdash & 0 & \rlongdash \end{pmatrix} 
			+ \begin{pmatrix} 0 & \llongdash & 0 & \rlongdash \\ \vdots & \vdots & \vdots & \vdots \\ \alpha_{n - m} & \llongdash & \beta_{n - m} + \alpha_{n - m} \beta_{n - m + 1} & \rlongdash \end{pmatrix}
		\end{align*}
		
		for \emph{all} choices of $\alpha_1, \ldots, \alpha_{n - m}$. \\

		The same reasoning can be applied to other charts (i.e. other choices of $g \in GL_{n - m + 1}$) by replacing $\begin{pmatrix} 0 & \llongdash & \gamma_1 & \rlongdash \\ \vdots & \vdots & \vdots & \vdots \\ 1 & \llongdash & \gamma_{n - m + 1} & \rlongdash \end{pmatrix}$ with $\begin{pmatrix} \llongdash & \gamma_1 & \rlongdash \\ \vdots & \vdots & \vdots \\ \llongdash & \gamma_{n - m + 1} & \rlongdash \end{pmatrix} = g \cdot \begin{pmatrix} 0 & \llongdash & \beta_1 & \rlongdash \\ \vdots & \vdots & \vdots & \vdots \\ 1 & \llongdash & \beta_{n - m + 1} & \rlongdash \end{pmatrix}$ for appropriate $\gamma_i \in \mathbb{A}^{m + 1}$. \\
		
		Before computing the dimension of solutions to explicit polynomial equations, we will consider heuristics from expected dimensions of Fano varieties of $k$-planes contained in general hypersurfaces. For example, suppose that $n - m = 2$. Fixing $\gamma_1$ and $\gamma_2$, the solutions $ \gamma_3$ to $F(\gamma_1 + \alpha_1 \gamma_3, \gamma_2 + \alpha_2 \gamma_3) = 0$ for all $\alpha_1, \alpha_2$ correspond to planes in $\mathbb{A}^m$ contained in the intersection of a hypersurface in $\mathbb{A}^{2m}$ with the complete intersection of hypersurfaces of the form $\widetilde{H}_j = x_1 x_{m + j} - x_j x_{m + 1}$ for $2 \le j \le m + 3$. If this is a complete intersection with $(F = 0)$, then the fact that planes in $\mathbb{A}^{2m + 2}$ correspond to lines in $\mathbb{P}^{2m + 1}$ implies that the expected dimension of lines  (p. 4 of \cite{CZ}) contained in this complete intersection is $2(2m + 1) - 3(m - 1) - c = m - c + 5$, where $c$ is the degree of $F$ as a polynomial in $2m + 2$ variables. However, the condition that the lines are of the type $(x, 0) + (0, x) \cdot \beta$ gives a codimension $m$ condition and we would generically expect the set to be empty for sufficiently large $c$. \\

		Starting with a fixed $\gamma_1, \gamma_2$ as above, we can work out the (usually) codimension $m$ condition on $\gamma_3$ more explicitly. Again, we would like to find $x \in \mathbb{A}^m$ such that $F(x, \beta x) = 0$ for all $\beta$. This boils down to coefficients in using terms involving $\beta$ being set equal to $0$. Generically, this reduces the dimension by $e$, where $e$ is the degree of $F$ with respect to the final $m$ coefficients. In general, the equations involved can be analyzed using the Taylor expansion of $F$ at a particular point. Given $\gamma_i = (\gamma_{i, 1}, \ldots, \gamma_{i, m}) \in \mathbb{A}^m$, we study solutions to 
		
		\[ F((\gamma_1, \ldots, \gamma_r) + (\alpha_1 \gamma_{r + 1}, \ldots, \alpha_r \gamma_{r + 1})) = F(\gamma_1, \ldots, \gamma_r) + \sum_{i, j} \frac{\partial F}{\partial x_{ij}}(\gamma)  (\alpha_i \gamma_{r + 1, j}) \] \[ + \frac{1}{2!} \sum_{i, j, k, l} \frac{\partial^2 F}{\partial x_{ij} \partial x_{kl}}(\gamma) (\alpha_i \gamma_{r + 1, j})(\alpha_k \gamma_{r + 1, l}) + \frac{1}{3!} \sum_{i, j, k, l, p, q} \frac{\partial^3 F}{\partial x_{ij} \partial x_{kl} \partial x_{pq}}(\gamma) (\alpha_i \gamma_{r + 1, j})(\alpha_k \gamma_{r + 1, l})(\alpha_p \gamma_{r + 1, q})  + \ldots = 0 \]

		 which hold for all $\alpha_1, \ldots, \alpha_r$. If $\alpha_1 \ne 0$, we can assume without loss of generality that $\alpha_1 = 1$. Note that there will be a total of $\deg F$ sums.  \\ 
		 
		 Interpreting $F$ as a polynomial in the $\alpha_i$ with coefficients which are polynomials in the $\gamma_{i, j}$, we need all the coefficients in $\gamma_{r, s}$ to be equal to $0$. Each term is a sum of the form \[ \sum_{\substack{1 \le i_a \le r \\ 1 \le j_b \le m + 1 } } \frac{\partial^u F}{\partial x_{i_1 j_1} \cdots \partial x_{i_u j_u} }(\gamma) (\alpha_{i_1} \gamma_{r + 1, j_1}) \cdots(\alpha_{i_u} \gamma_{r + 1, j_u} ). \] This gives the degree $u$ terms as a polynomial in the $\alpha_i$. Now consider the degree $1$ term. The coefficient of $\alpha_i$ being $0$ requires $m + 1$ polynomials to vanish. Repeating this for each $i$ already gives a total of $r(m + 1)$ conditions. Since all the other coefficients are also equal to $0$, the set of solutions is empty for a generic choice of $F$.

	\end{proof}

	\begin{rem}
		Given a particular $(n - m)$-plane $\Gamma \in \mathbb{G}(n - m, n)$, the space of $(n - m - 1)$-planes contained in $\Gamma$ forms an $(n - m)$-dimensional linear subspace of $\mathbb{G}(n - m - 1)$ (Theorem 3.16 on p. 110 and proof of Theorem 3.20 (ii) on p. 114 of \cite{HT}). Thus, the space of $(n - m)$-planes in $\mathbb{P}^n$ whose $(n - m - 1)$-dimensional linear subspaces are contained in $Z$ is contained in the space of maximal $(n - m)$-planes in $\mathbb{G}(n - m - 1, n)$ which are contained in the hypersurface $Z \subset \mathbb{G}(n - m - 1, n)$. With this interpretation, there are a couple more options for genericity conditions which imply $Z$ is empty. 
		
		\begin{enumerate}
			\item Let $M$ be the subvariety of $F_{n - m}(\mathbb{G}(n - m - 1, n)) \subset \mathbb{G}(n - m, N)$ consisting of $(n - m)$-planes in $\mathbb{G}(n - m - 1, n)$ maximal with respect to inclusion. If $M$ is a general $GL_{N + 1}$-translate of $M$ and $Z$ is contained in some hypersurface of degree $e$ in $\mathbb{P}^N$ of sufficiently large degree, then the subvariety of $\mathbb{G}(n - m, n)$ consisting of $(n - m)$-planes $\Gamma$ such that $\Lambda \in Z$ for all $(n - m - 1)$-planes $\Lambda \subset \Gamma$ is empty if $e$ is sufficiently large compared to $n - m$ by Kleiman's transversality theorem (Theorem on p. 290 of \cite{Kl}) while taking $\mathbb{G}(n - m, N) = G(n - m + 1, N + 1)$ to be a homogeneous space with a transitive $GL_{N + 1}$-action.

			\item We can follow the usual proof of the generic dimension estimates of Fano varieties of $k$-planes to show that the subvariety of $\mathbb{G}(n - m, n)$ consisting of $(n - m)$-planes $\Gamma$ such that $\Lambda \in Z$ for all $(n - m - 1)$-planes $\Lambda \subset \Gamma$ is empty if $n \gg 0$ and $Z$ is contained in \emph{some} generic hypersurface $A \subset \mathbb{P}^N$ \emph{not} containing $\mathbb{G}(n - m - 1, n)$. \\
			
			Let $\mathbb{G}(n - m - 1, n) \hookrightarrow \mathbb{P}^N$ be the Pl\"ucker embedding and $\mathbb{P}^M$ with $M = \binom{N + e}{e} - 1$ be the space of degree $e$ hypersurfaces in $\mathbb{P}^N$. Let \[ \Phi = \{ (\Gamma, A) \in \mathbb{G}(n - m, N) \times \mathbb{P}^M : \Gamma \subset A \cap \mathbb{G}(n - m - 1, n), \mathbb{G}(n - m - 1, n) \not\subset A \}. \]
				
			Consider the projections
				
				\begin{center}
					\begin{tikzcd}
						& \Phi \arrow{dr}{\psi} \arrow{dl}[swap]{\varphi} \\
						\mathbb{G}(n - m, N)  & & \mathbb{P}^M. \\
					\end{tikzcd}
				\end{center}
				
				Then, we have that $\psi^{-1}(A) \cong F_{n - m}(A \cap \mathbb{G}(n - m - 1, n))$ and \[ \varphi^{-1}(\Gamma) = \{ A \in \mathbb{P}^M : \Gamma \subset A \cap \mathbb{G}(n - m - 1, n), \mathbb{G}(n - m - 1, n) \not\subset A \} \]

			for each $\Gamma \in \varphi(\Phi)$. Note that $\varphi(\Phi)$ consists of $(n - m)$-planes in $\mathbb{P}^N$ which are contained in $A \cap \mathbb{G}(n - m - 1, n)$ for some degree $e$ hypersurface $A \subset \mathbb{P}^N$.  \\
						
			Fix $\Gamma \in \varphi(\Phi)$. The $(n - m)$-plane $\Gamma$ in $\mathbb{P}^N$ is contained in $A \cap \mathbb{G}(n - m - 1, n)$ for \emph{some} degree $e$ hypersurface $A \subset \mathbb{P}^N$ \emph{not} containing $\mathbb{G}(n - m - 1, n)$ if and only if there is some $f \in H^0(\mathbb{G}(n - m - 1, n), \mathcal{O}_{\mathbb{G}(n - m - 1, n)}(e))$ such that $f|_\Gamma = 0$. In other words, $f$ is in the kernel of the restriction map $\rho : H^0(\mathbb{G}(n - m - 1, n), \mathcal{O}_{\mathbb{G}(n - m - 1, n)}(e)) \longrightarrow H^0(\Gamma, \mathcal{O}_\Gamma(e))$. Note that this map is surjective since we assumed that $\Gamma$ is contained in $\mathbb{G}(n - m - 1, n)$. Since $\Gamma \cong \mathbb{P}^{n - m}$, we have that $\dim H^0(\Gamma, \mathcal{O}_\Gamma(e)) = \binom{n - m + e}{e}$. On the other hand, we have that \[ \dim H^0(\mathbb{G}(n - m - 1, n), \mathcal{O}_{\mathbb{G}(n - m - 1, n)}(e)) = \prod_{j = n - m + 1}^{n + 1} \frac{\binom{e + j - 1}{e}}{\binom{e + j - (n - m) - 1}{e}} \]
			
			since the determinant of the dual of the tautological bundle is the pullback of $\mathcal{O}_{\mathbb{G}(n - m - 1, n)}(1)$ by the Pl\"ucker embedding (Proposition 5.2 on p. 388 of \cite{EF}). This means that \[ \dim \varphi^{-1}(\Gamma) = \prod_{j = n - m + 1}^{n + 1} \frac{\binom{e + j - 1}{e}}{\binom{e + j - (n - m) - 1}{e}} - \binom{n - m + e}{e} \]
			
			for each $\Gamma \in \varphi(\Phi) \subset \mathbb{G}(n - m, N)$. \\
			
			Given suitable parameters, we have that $\dim \Phi < M$ and a generic element of $\mathbb{P}^M$ is \emph{not} in the image of $\psi : \Phi \longrightarrow \mathbb{P}^M$.

		\end{enumerate}
	\end{rem}
	
	The assumptions of Proposition \ref{smallrem} will be denoted using the following term. 
	
	\begin{defn}\label{lingen}
		A variety $Y \subset \mathbb{P}^n$ is \emph{$k$-linearly generic} if the locus of non-generic $k$-planes in $\mathbb{P}^n$ in the sense of Proposition \ref{smallrem} is contained in a hypersurface generic in its degree.
	\end{defn}
	
	Using Proposition \ref{smallrem} and Proposition \ref{lam1}, we can be bound the relative dimension (Definition \ref{reldimdef}) of $\widetilde{B_2}$.
	
	\begin{prop}  \label{bbound}
		If $Y \subset \mathbb{P}^n$ is $u$-linearly generic for $u \le k - 2$, the relative dimension (Definition \ref{reldimdef}) of $\dim \widetilde{B}_2 \le - 2(n - m - (k - 2) - 1)$ if $k - 2 \le n - k + 2$.
	\end{prop}
	
	\begin{proof}
		Since we assume that $Y \subset \mathbb{P}^n$ is $u$-linearly generic for $u \le k - 2$, we can assume that the $(\lambda + 1)$-planes in question always contain some generic $\lambda$-plane. Let \[ S = \{ (\Lambda, \Gamma) \in \mathbb{G}(\lambda + 1, n) \times \mathbb{G}(n - m, n) : \Lambda \subset \Gamma, \Lambda \in Q, |\Gamma \cap Y| = d \}. \]
		
		Consider the projections
		
		\begin{center}
			\begin{tikzcd}
				& S \arrow{dr}{\psi} \arrow{dl}[swap]{\varphi} \\
				\mathbb{G}(\lambda + 1, n)  & & \mathbb{G}(n - m, n). 
			\end{tikzcd}
		\end{center}
		
		In this diagram, the space of $(n - m)$-planes containing some element of $Q$ that is \emph{not} tangent to $Y$ is given by $\psi(S)$ and $Q = \varphi(S)$ (writing $Q$ to mean $\psi(Q)$ from Proposition \ref{lam1}). Thus, it suffices to give an upper bound for $\dim \psi(S)$. Working over each irreducible component of $S$, Theorem 11.12 and Corollary 11.13 on p. 138 -- 139 of \cite{H1} imply that $\dim S \ge \dim \psi(S) + M$ if $\dim \psi^{-1}(\Gamma) \ge M$ for each $\Gamma \in \psi(S)$. Rearranging this inequality gives the upper bound $\dim \psi(S) \le \dim S - M$. \\
		
		If we fix $\Lambda \in Q = \varphi(S)$, we have that $\varphi^{-1}(\Lambda) \subset \mathbb{G}(n - m - \lambda - 2, n - \lambda - 2)$ is a nonempty open subset for each $\Lambda \in \varphi(S) = Q$. Since $\mathbb{G}(n - m - \lambda - 2, n - \lambda - 2)$ is irreducible, this is a dense open subset and $\dim \varphi^{-1}(\Lambda) = \dim \mathbb{G}(n - m - \lambda - 2, n - \lambda - 2)$ for each $\Lambda \in \varphi(S)$. Although the fibers can be more complicated for $\psi$, we can still find a (relatively) uniform method of bounding the dimension. \\
		
		Given a fixed $(n - m)$-plane $\Gamma \in \psi(S)$, we have that $\psi^{-1}(\Gamma)$ consists of $(\lambda + 1)$-planes $\Lambda$ such that $\Lambda \subset \Gamma$ and $|\Lambda \cap Y| \ge k + 1$. In other words, we are looking for $(\lambda + 1)$-planes contained in $\Lambda \cong \mathbb{P}^{n - m}$ that intersect $Y$ in $\ge k + 1$ points. Note that $Y \cap \Lambda \subset Y \cap \Gamma$. Since we take these $k$-planes to be contained in $Y$ and $Y \cap \Lambda \subset Y \cap \Gamma$, we only have finitely many choices for their points of intersection with $Y$. In particular, we can express $\psi^{-1}(\Gamma)$ as the union of elements containing each $(k + 1)$-tuple in $Y \cap \Gamma$. Given an unordered $(k + 1)$-tuple of points in $Y \cap \Gamma$, let $T_p$ be the elements of $\psi^{-1}(\Gamma)$ containing $p$. This implies that \[ \psi^{-1}(\Gamma) = \bigcup_p T_p
		\Rightarrow \dim \psi^{-1}(\Gamma) = \max_p \dim T_p, \] where $p$ varies over $(k + 1)$-tuples of points in $Y \cap \Gamma$ which span a linear subspace of dimension $\mu \le \lambda + 1$. These $(k + 1)$-tuples can be further partitioned into locally closed subspaces corresponding to $(k + 1)$-tuples spanning a linear subspace of a given dimension. Since there is a finite number of possible dimensions, it suffices to look at individual $(k + 1)$-tuples and take the maximum dimension. \\
		
		Given a fixed $(k + 1)$-tuple in $Y \cap \Gamma$ spanning a linear subspace of dimension $\mu \le \lambda + 1$, the space of $(\lambda + 1)$-planes in $\Gamma \cong \mathbb{P}^{n - m}$ which contain these points is isomorphic to $\mathbb{G}(\lambda + 1 - \mu - 1, n - m - \mu - 1) = \mathbb{G}(\lambda - \mu, n - m - \mu - 1)$. We actually have a lower bound for the space of such $(\lambda + 1)$-planes since
		\begin{align*}
			\dim \mathbb{G}(\lambda - \mu, n - m - \mu - 1) &= (n - m - \lambda - 1)(\lambda - \mu + 1) \\
			&\ge n - m - \lambda - 1. 
		\end{align*}
		
		Since this lower bound does not depend on $\mu$, it applies to \emph{any} $(k + 1)$-tuple of points $p$. Thus, we have that $\dim \psi^{-1}(\Gamma) \ge n - m - \lambda - 1$ for each $\Gamma \in \psi(S)$ and we can set $M = n - m - \lambda - 1$ above. By Proposition \ref{lam1} and Theorem \ref{seclen}, this implies that 
		\begin{align*}
			\dim \psi(S) &\le \dim S - M \\
			&= \dim S - (n - m - \lambda - 1) \\
			&= \dim Q + \dim \mathbb{G}(n - m - \lambda - 2, n - \lambda - 2) - (n - m - \lambda - 1) \\
			&= \dim \mathbb{G}(\lambda, n) + \dim Y_{k + 1}^\lambda - (\lambda + 1) + \dim \mathbb{G}(n - m - \lambda - 2, n - \lambda - 2) - (n - m - \lambda - 1) \\
			&= (\lambda + 1)(n - \lambda) + (n - \lambda - 1) - (k + 1)((n - m) - \lambda - 1) - (\lambda + 1) + m(n - m - \lambda - 1) - (n - m - \lambda - 1) \\
			&= (\lambda + 1)(n - \lambda) - (k + 1)((n - m) - \lambda - 1) + m(n - m - \lambda - 1) + (m - \lambda - 1). 
		\end{align*}
		
		Thus, the space of $(n - m)$-planes containing a $(\lambda + 1)$-plane intersecting $Y$ at $\ge k + 1$ points which contains \emph{some} generic $\lambda$-plane has dimension at most \[ (\lambda + 1)(n - \lambda) - (k + 1)((n - m) - \lambda - 1) + m(n - m - \lambda - 1) + (m - \lambda - 1). \] This implies the same bound for those which intersect $Y$ at exactly $k + 1$ points. \\

		Let \[ D = (\lambda + 1)(n - \lambda) - (k + 1)((n - m) - \lambda - 1) + m(n - m - \lambda - 1) + (m - \lambda - 1). \] The relative dimension of these $(n - m)$-planes in $\widehat{\mathcal{K}}$ is 
		\begin{align*}
			D - m(n - m + 1) &= (\lambda + 1)(n - \lambda) - (k + 1)((n - m) - \lambda - 1) + m(n - m - \lambda - 1) + (m - \lambda - 1) -m(n - m + 1) \\
			&= (\lambda + 1)(n - m - \lambda) - (k + 1)(n - m - \lambda) + (k + 1) - (\lambda + 1) \\
			&= -(k - \lambda)(n - m - \lambda) + (k - \lambda) \\
			&= -(k - \lambda)(n - m - \lambda - 1) \\
			&\le - 2(n - m - (k - 2) - 1) 
		\end{align*}
		
		since $\lambda \le k - 2$ and $k + 1 \le n - m - 1$.

	\end{proof}
	
	\vspace{2mm}
	
	The same reasoning with $d - k - 2$ replacing $k$ implies the following bound for upper bound for the dimension of $\widetilde{T}_2 \subset V$. 
	
	\begin{prop} \label{multcont}
		If $Y \subset \mathbb{P}^n$ is $u$-linearly generic for $u \le d - k - 2$, the relative dimension $\dim \widetilde{T}_2 \le  -2(n - m - (d - k - 4) - 1)$. 
	\end{prop}
	
	Here is a summary of dimensions of the degeneracy loci: 
	
	\begin{multicols}{2}
		\begin{itemize}
			\item $\dim N \le m(d - k - 2) - 1$ (Proposition \ref{roughbd}) 
			
			\item $\dim M \le mk - 1$ (Proposition \ref{roughbd})

			\item $\dim P \le m(n - m + 1) - m - 1 + (d - k - 1)$ (Proposition \ref{tandim})

			\item $ \dim Q \le m(n - m + 1) - m - 1 + (k + 1)$ (Proposition \ref{tandim})

			\item $\dim \widetilde{B}_2 \le  - 2(n - m - (k - 2) - 1)$ as a \emph{relative} dimension in $\widehat{\mathcal{K}}$ (Proposition \ref{bbound}) 
			
			\item $\dim \widetilde{T}_2 \le  - 2(n - m - (d - k - 4) - 1)$ as a \emph{relative} dimension in $\widehat{\mathcal{K}}$ (Proposition \ref{multcont}) 
			
			\item $\dim C = (n - m)k + k - 1$ (Lemma \ref{rkdim}) 
			
			\item $\dim D = (n - m)(d - k - 2) + (d - k - 2) - 1$ (Lemma \ref{rkdim}) 
		\end{itemize}
	\end{multicols}
	
	The remaining terms to analyze are $C \subset (\mathbb{P}^{n - m})^{(k + 1)}$ and $D \subset (\mathbb{P}^{n - m})^{(d - k - 1)}$ of linearly dependent $(k + 1)$-tuples and $(d - k - 1)$-tuples of $\mathbb{P}^{n - m}$. By Lemma \ref{rkdim}, we have that $\dim C = (n - m)k + k - 1$ and $\dim D = (n - m)(d - k - 2) + (d - k - 2) - 1$. \\

	\begin{prop}\label{recursion}
		In $K_0(\Var_k)$, the classes $[C]$ and $[D]$ are polynomials in $\mathbb{L}$. 
	\end{prop}
	
	\begin{proof}
		We will show this by finding a recursive formula. Given $u \le r$, let $I_{u, n, r} \subset (\mathbb{P}^n)^{(r)}$ be the locally closed subset of $r$-tuples of points of $\mathbb{P}^n$ which form the columns of an $(n + 1) \times r$ matrix of rank $u$.
		We claim that $[I_{u, n, r}] = [\mathbb{G}(u - 1, n)][I_{u, u - 1, r}]$ in $K_0(\Var_k)$. The idea is to fix the linear subspace spanned by the columns of the matrix and consider coordinates of the columns with respect to a fixed basis for this linear subspace. We can either use a piecewise trivial fibration from a morphism sending the $r$-tuples of points to their span or form a morphism inducing a bijection of rational points. For each $\Lambda \in \mathbb{G}(u - 1, n)$, let $A_\Lambda$ be a $(n + 1) \times u$ matrix whose columns form a basis of $\Lambda$. Consider the morphism $\pi : \mathbb{G}(u - 1, n) \times I_{u, u - 1, r} \longrightarrow I_{u, n, r}$ defined by $(\Lambda, B) \mapsto A_\Lambda \cdot B$, where $B$ is taken under quotients by permutations of columns and division of the columns by nonzero scalars. \\ 
		
		Since the columns of $A_\Lambda$ are linearly independent and $B$ has $u$ linearly independent columns, the span of $A_\Lambda \cdot B$ is $\Lambda$. Since two identical matrices have the same span and the columns of $A_\Lambda$ are linearly independent, the map $\pi$ is injective on $k$-rational points. The surjectivity of $\pi$ comes from setting $\Lambda$ to be the span of an element of $C \in I_{u, n, r}$ and $B$ to be the matrix whose columns (up to quotienting) are the coordinates of the columns of $C$ with respect to the columns of $A_\Lambda$. Thus, $\pi$ induces a bijection on $k$-rational points and Proposition \ref{ratbij} implies that $[I_{u, n, r}] = [\mathbb{G}(u - 1, n)][I_{u, u - 1, r}]$ in $K_0(\Var_k)$.  \\

		By definition, we have that
		
		\begin{equation} \label{rec}
			[I_{u, u - 1, r}] = [(\mathbb{P}^{u - 1})^{(r)}] - \sum_{v = 1}^{u - 1} [I_{v, u - 1, r}].
		\end{equation}
		
		For each $1 \le v \le u - 1$, the same reasoning as above implies that $[I_{v, u - 1, r}] = [\mathbb{G}(v - 1, u - 1)][I_{v, v - 1, r}]$ and \[ [I_{v, v - 1, r}] = [(\mathbb{P}^{v - 1})^{(r)}] - \sum_{w = 1}^{u - 1} [I_{w, v - 1, r}].  \]
		
		In each step of this recursion, the indices $a, b$ in $I_{a, b, r}$ are strictly smaller than those in the previous step. So, this process must stop after a finite number of steps. Since $[I_{1, b, r}] = [\mathbb{P}^b]$ and $[I_{2, b, r}] = [\mathbb{G}(1, b)][I_{2, 1, r}] = [\mathbb{G}(1, b)]([(\mathbb{P}^1)^{(r)}] - [\mathbb{P}^1])$, the reduction $[I_{u, n, r}] = [\mathbb{G}(u - 1, n)][I_{u, u - 1, r}]$ followed by induction on $u$ in $I_{u, u - 1, r}$ via the recursion \ref{rec} implies that $[I_{u, n, r}]$ is a polynomial in $\mathbb{L}$ for each $u \le r \le n + 1$. \\
		
		Since $[C] = [(\mathbb{P}^{n - m})^{(k + 1)}] - [I_{k + 1, n - m, r}]$ and $[D] = [(\mathbb{P}^{n - m})^{(d - k - 1)}] - [I_{d - k - 1, n - m, r}]$, they must also be polynomials in $\mathbb{L}$. 
	\end{proof}
	
	\vspace{1mm}

	\begin{rem} \label{recursionrem}
		\begin{enumerate}
			\item The degrees of polynomials in $\mathbb{L}$ giving the classes of $C$ and $D$ in $K_0(\Var_k)$ are given by $\dim C = (n - m)k + k - 1$ and $\dim D = (n - m)(d - k - 2) + (d - k - 2) - 1$.  \\
			
			\item The coefficients of $\mathbb{L}^k$ can be expressed in terms of multinomial coefficients and sizes of partitions corresponding to certain Young tableaux (Example 2.4.5 on p. 72 -- 73). These come from the classes of symmetric products $(\mathbb{P}^a)^{(b)}$ and Grassmannians $\mathbb{G}(c, d)$ respectively. 
		\end{enumerate}
	\end{rem}

	Next, we use the computations earlier in this section to find dimensions of terms of degeneracy loci in the expression

	\begin{align}
		\frac{2[F_{n - m}(Y)]([(\mathbb{P}^{n - m})^{(k + 1)}] - [(\mathbb{P}^{n - m})^{(d - k - 1)}])}{[\mathbb{G}(n - m, n)]} = \underbrace{\frac{([Y^{(k + 1)}] - [M])[G(n - m + 1 - (k + 1), n + 1 - (k + 1))]}{[\mathbb{G}(n - m, n)]}}_{\text{Term 1}} \label{init1} \\
		- \underbrace{\frac{([Y^{(d - k - 1)}] - [N])[G(n - m + 1 - (d - k - 1), n + 1 - (d - k - 1))]}{[\mathbb{G}(n - m, n)]}}_{\text{Term 2}}  + \underbrace{\frac{[P] - [Q]}{[\mathbb{G}(n - m, n)]}}_{\text{Term 3}} + \underbrace{\frac{[\widetilde{B}_2] - [\widetilde{T}_2]}{[\mathbb{G}(n - m, n)]}}_{\text{Term 4}} \label{init2} \\
		+ \underbrace{\frac{2[F_{n - m}(Y)]([C] - [D])}{[\mathbb{G}(n - m, n)]}}_{\text{Term 5}}. \label{fano} \\
		\nonumber
	\end{align}
	
	Here, we will take ``degeneracy loci'' to be non-generic subsets of incidence correspondences involved in the simplified higher dimensional $Y-F(Y)$ relation. Let $\alpha = \dim \mathbb{G}(n - m, n) = m(n - m + 1)$. Substituting in the (upper bounds of) dimensions of the degeneracy loci above yields the following dimensions in $\widehat{\mathcal{K}}$:

	\begin{itemize}
		\item Terms 1 and 2 from \ref{init1} and \ref{init2}: 
		\begin{itemize}
			\item Main terms $[Y^{(k + 1)}][G(n - m + 1 - (k + 1), n + 1 - (k + 1))]$ and $[Y^{(d - k - 1)}][G(n - m + 1 - (d - k - 1), n + 1 - (d - k - 1))]$: \begin{align*}
				\dim Y^{(k + 1)} + \dim G(n - m + 1 - (k + 1), n + 1 - (k + 1)) - \alpha &= m(k + 1) + m(n - m - k) - m(n - m + 1) \\
				&= m(k + 1 + n - m - k - n + m - 1) \\
				&= 0
			\end{align*}
			\begin{align*}
				\dim Y^{(d - k - 1)} + \dim G(n - m + 1 - (d - k - 1), n + 1 - (d - k - 1)) - \alpha &= m(d - k - 1) \\
				&+ m(n - m + 1 - (d - k - 1)) \\
				&- m(n - m + 1) \\
				&= m(n - m + 1) - m(n - m + 1) \\
				&= 0
			\end{align*}
			\item Degenerate terms $[M][G(n - m + 1 - (k + 1), n + 1 - (k + 1))]$ and $[N][G(n - m + 1 - (d - k - 1), n + 1 - (d - k - 1))]$: \begin{align*}
				\dim M + \dim G(n - m + 1 - (k + 1), n + 1 - (k + 1)) - \alpha &\le mk - 1 + m(n - m - k) - m(n - m + 1) \\
				&= m(k + n - m - k - n + m - 1) - 1 \\
				&= m(-1) - 1 \\
				&= -m - 1 
			\end{align*}
			\begin{align*}
				\dim N + \dim G(n - m + 1 - (d - k - 1), n + 1 - (d - k - 1)) - \alpha &\le m(d - k - 2) - 1 \\ 
				&+ m(n - m + 1 - (d - k - 1)) \\
				&- m(n - m + 1) \\
				&= m((d - k - 2) + (n - m + 1) - (d - k - 1) \\
				&- (n - m + 1)) - 1 \\
				&= - m - 1
			\end{align*}

		\end{itemize}

		\item Term 3 (tangent planes) from \ref{init2} \begin{align*}
			\dim P - \alpha &\le S = m(n - m + 1) - m - 1 + (d - k - 1)- m(n - m + 1)  \\
			&= - m - 1 + (d - k - 1)
		\end{align*}
		
		since $n - m > 2m$ under the conditions of Theorem \ref{avglims}. 
	
		The same reasoning with Proposition \ref{tandim} implies that \[ \dim Q - \alpha \le - m - 1 + (k + 1). \]

		\item Term 4 (degenerate incidence correspondences) from \ref{init2}: \\
		
		By Proposition \ref{bbound} and Proposition \ref{multcont}, we have that \[\dim \widetilde{B}_2 - \alpha \le  -2(n - m - (k - 2) - 1) \] and \[ \dim \widetilde{T}_2 - \alpha \le  -2(n - m - (d - k - 4) - 1). \]

		\item Term 5 (degeneracies involving $F_{n - m}(Y)$) from \ref{fano}:
		If $Y$ is contained in a smooth hypersurface $X \subset \mathbb{P}^n$ of degree $r$, Theorem 4.3 on p. 266 of \cite{Ko3} implies that 
		
		\begin{align*}
			\dim F_{n - m}(Y) + \dim C - \alpha &\le \dim F(Y) + \dim C - \alpha \\
			&= 2n - 3 - r + (n - m)k + k - 1 - m(n - m + 1) \\
			&= (n - m)(k - m + 1) + n + k - r - 4 \\
			&= -(n - m)(m - k - 1) + n + k - r - 4 \\
			&\le -(n - m)(m - k - 1) + n + k 
		\end{align*}
		
		and \begin{align*}
			\dim F_{n - m}(Y) + \dim D - \alpha &\le \dim F(Y) + \dim C - \alpha \\
			&= 2n - 3 - r + (n - m)(d - k - 2) + (d - k - 2) - 1 - m(n - m + 1) \\
			&= (n - m)((d - k - 2) - m + 1) + n + (d - k - 2) - r - 4 \\
			&= -(n - m)(m - (d - k - 2) - 1) + n + (d - k - 2) - r - 4 \\
			&\le -(n - m)(m - (d - k - 2) - 1) + n + (d - k - 2). 
		\end{align*}

		\item Variable size restrictions:
		\begin{multicols}{3}
			\begin{itemize}			
				\item $d \ge k + 3$ 
				
				\item $d - k - 1 \le n - m - 1$ 
				
				\item $k + 1 \le n - m - 1$ 
				
				\item $n - m \le m - 1$ 
				
				\item $d \ge (n - m) + 2$ 
			\end{itemize}
		\end{multicols}

	\end{itemize}
	
	\subsubsection{Higher degree varieties ($d - k - 1 > n - m - 1$)} \label{highdeg}
	
	Most of the ideas in Section \ref{lowdeg} carry over for the dimension estimates in the case where $d - k - 1 > n - m - 1$. The key difference is that the extended $Y-F(Y)$ relation involves different sets since a generic $(d - k - 1)$-tuple lying on an $(n - m)$-plane is not linearly independent, but spans a linear subspace of dimension $n - m$ in $\mathbb{P}^n$. Let \[ J = \{ ((p_1, \ldots, p_{d - k - 1}), \Lambda) \in V : \text{ $p_i$ distinct, } \dim \overline{p_1, \ldots, p_{d - k - 1}} = n - m \} \] and $\widetilde{J} \subset J$ be the subset where $\Lambda \not\subset Y$. Note that $J = V \setminus \widetilde{T}_1$. Finally, let $\widetilde{T}_{11} \subset \widetilde{T}_1$ be the subset with $\Lambda \not\subset Y$ and $T_{12} \subset \widetilde{T}_1$ be the subset with $\Lambda \subset Y$. In the notation below, we have that $[T_{12}] = [F_{n - m}(Y)][D]$ in $K_0(\Var_k)$. \\
	
	The setup in Section \ref{lowdeg} (p. 17 -- 18) implies that
	\begin{align*}
		([Y^{(k + 1)}] - [M])[G(n - m + 1 - (k + 1), n + 1 - (k + 1))] - [Q] - [\widetilde{T}_2] - 2[F_{n - m}(Y)]([(\mathbb{P}^{n - m})^{(k + 1)}] - [C]) \\
		= [V] - [\widetilde{T}] - [\widetilde{R}] \\
		= ([V] - [\widetilde{T}_1]) - [\widetilde{T}_2] - [\widetilde{R}] \\
		= ([V] - [\widetilde{T}_1] + [Q] + [\widetilde{R}]) - [\widetilde{R}] - [Q] - [\widetilde{T}_2] - [\widetilde{R}]
	\end{align*}

	Taking this into account and using the proof of Proposition \ref{extend} for the variables listed below gives the following expression in $\widehat{\mathcal{K}}$: 
	
	\begin{align}
		\frac{2[F_{n - m}(Y)]([(\mathbb{P}^{n - m})^{(k + 1)}] - [\UConf_{d - k - 1} \mathbb{P}^{n - m} ])}{[\mathbb{G}(n - m, n)]} = \underbrace{\frac{([Y^{(k + 1)}] - [M])[G(n - m + 1 - (k + 1), n + 1 - (k + 1))]}{[\mathbb{G}(n - m, n)]}}_{\text{Term 1}} \label{hinit1} \\
		- \underbrace{\frac{[J]}{[\mathbb{G}(n - m, n)]}}_{\text{Term 2}} - \underbrace{\frac{[Q]}{[\mathbb{G}(n - m, n)]}}_{\text{Term 3}} + \underbrace{\frac{[\widetilde{B}_2] - [\widetilde{T}_2]}{[\mathbb{G}(n - m, n)]}}_{\text{Term 4}} + \underbrace{\frac{2[F_{n - m}(Y)]([C] - [D])}{[\mathbb{G}(n - m, n)]}}_{\text{Term 5}} \label{hinit2} \\
		\nonumber
	\end{align}
	
	\vspace{-10mm}
	
	where \[ J = \{ ((p_1, \ldots, p_{d - k - 1}), \Lambda) \in V : \text{ $p_i$ distinct, } \dim \overline{p_1, \ldots, p_{d - k - 1}} = n - m \}. \]

	Note that $J = W \setminus \widetilde{B}_1$ using the definition of $\widetilde{B}_1$ below. In this higher degree setting, we take 
	
	\begin{itemize}
		\item $D \subset (\mathbb{P}^{n - m})^{(d - k - 1)}$ is the set of $(d - k - 1)$-tuples of distinct points spanning a linear subspace of dimension $\le n - m - 1$. This can also be embedded inside $\UConf_{d - k - 1} \mathbb{P}^{n - m}$, where $\UConf_r X \subset X^{(r)}$ denotes unordered $r$-tuples of distinct points on $X$.
		
		\item $C \subset (\mathbb{P}^{n - m})^{(k + 1)}$ is the set of linearly dependent $(k + 1)$-tuples of points in $\mathbb{P}^{n - m}$ 
		
		\item $\widetilde{B} = \widetilde{B}_1 \sqcup \widetilde{B}_2$, where \[ \widetilde{B}_1 = \{ ((p_1, \ldots, p_{d - k - 1}), \Lambda) \in W : \text{ $p_i$ distinct, },  \dim \overline{p_1, \ldots, p_{d - k - 1}} \le n - m - 1 \} \] and
		\begin{align*}
			\widetilde{B}_2 &= \{ ((p_1, \ldots, p_{d - k - 1}), \Lambda) \in W : \text{ $p_i$ distinct, } \dim \overline{p_1, \ldots, p_{d - k - 1}} = n - m \text{ but } \\ 
			&(Y \cap \Lambda) \setminus \{ p_1, \ldots, p_{d - k - 1} \} \text{ \emph{not} a linearly independent $(k + 1)$-tuple, } \Lambda \not\subset Y \}
		\end{align*}
		
		\item $\widetilde{T} = \widetilde{T}_1 \sqcup \widetilde{T}_2$, where \[ \widetilde{T}_1 = \{ ((p_1, \ldots, p_{k + 1}), \Lambda) \in V : p_1, \ldots, p_{k + 1} \text{ linearly dependent} \}  \] and
		\begin{align*}
			\widetilde{T}_2 &= \{ ((p_1, \ldots, p_{k + 1}), \Lambda) \in V : p_1, \ldots, p_{k + 1} \text{ linearly independent but } \\ 
			&(Y \cap \Lambda) \setminus \{ p_1, \ldots, p_{k + 1} \} \text{ \emph{not} a $(d - k - 1)$-tuple spanning an $(n - m)$-dimensional linear subspace }, \Lambda \not\subset Y \} 
		\end{align*}
		
	\end{itemize}

	As in Section \ref{lowdeg}, our goal of this section is to compute the dimensions listed below. The relative dimensions in $\widehat{\mathcal{K}}$ are listed on p. 28 -- 29.

		\begin{itemize}
			\item $\dim M \le mk - 1$ (Proposition \ref{roughbd}) 
			
			\item $\dim Q \le m + (m - (n- m))(n - m)$ (Proposition \ref{tandim})
			
			\item $\dim \widetilde{B}_2 \le -2(n - m - (k - 2) - 1)$ as a \emph{relative} dimension in $\widehat{\mathcal{K}}$ (Proposition \ref{bbound}) 
			
			\item $\dim \widetilde{T}_2 = \emptyset$ (Proposition \ref{highmultcont}) 
			
			\item $\dim C = (n - m)k + k - 1$ (Lemma \ref{rkdim}) 
			
			\item $\dim D = (n - m - 1)(d - k) - (d - k - 1)$ 
			This follows from the proof of Lemma \ref{rkdim} in Section \ref{lowdeg}. 
		\end{itemize}

	It suffices to show that there are $m, n, d, k$ satisfying these inequalities along with the following variable restrictions: 
	
	\begin{multicols}{3}
		\begin{itemize}
			\item $d \ge k + 3$ 
			
			\item $k + 1 \le n - m - 1$ 
			
			\item $n - m \le m - 1$ (implies that $d - k - 1 > n - m - 1$)
			
			\item $d \ge (n - m) + 2$ 
		\end{itemize}
	\end{multicols}
	
	Since the variable restrictions are compatible with the setting of Proposition \ref{bbound}, we only need to compute a bound for the relative dimension of $\widetilde{T}_2$. This follows from repeating the same steps with a change in parameters.

	\begin{prop}\label{highmultcont}
		If $Y \subset \mathbb{P}^n$ is $u$-linearly generic for $u \le d - k - 1$ (Definition \ref{lingen}), the first part of Proposition \ref{lam1} implies that $\widetilde{T}_2 = \emptyset$.
	\end{prop}

	\begin{proof}
		In Theorem \ref{seclen}, we will take $\lambda \le n - m - 2$. Since we assumed that $d - k - 1 > n$ in Part \ref{highendlim} of Theorem \ref{avglims}, the locus in question is empty since the total space is $\mathbb{P}^{n - \lambda - 1}$ and the codimension is $(d - k - 1)((n - m) - \lambda - 1)$. The convention in Remark \ref{condns} implies that $\widetilde{T}_2 = \emptyset$.  
	\end{proof}
	Combining this with Proposition \ref{smallrem} and Proposition \ref{bbound}, we obtain the following dimensions for the degeneracy loci:

	\begin{itemize}
		\item $\dim M \le mk - 1$ (Proposition \ref{roughbd}) 
		
		\item $\dim Q \le m + (m - (n- m))(n - m)$ (Proposition \ref{tandim}) 
		
		\item $\dim \widetilde{B}_2 \le 2(n - m - (k - 2) - 1)$ as a \emph{relative} dimension in $\widehat{\mathcal{K}}$ (Proposition \ref{bbound}) 
		
		\item $\widetilde{T}_2 = \emptyset$ (Proposition \ref{highmultcont}) 
		
		\item $\dim C = (n - m)k + k - 1$ (Lemma \ref{rkdim}) 
		
		\item $\dim D = (n - m - 1)(d - k) - (d - k - 1)$ 
		This follows from the proof of Lemma \ref{rkdim} in Section \ref{lowdeg}. 
	\end{itemize}
	
	Before computing the relative dimensions, we write give a higher degree counterpart to Proposition \ref{recursion} for $(d - k - 1)$-tuples. 
	
	\begin{prop}\label{highrecursion}
		In $K_0(\Var_k)$, the classes $[C]$ and $[D]$ are polynomials in $\mathbb{L}$. 
	\end{prop}
	
	\begin{proof}
		Since $C$ is defined in the same way as the low degree case, it remains to consider $D$, which considers $(d - k - 1)$-tuples which aren't necessarily linearly independent. This means that we need to add the condition that the points of $\mathbb{P}^{n - m}$ corresponding to columns of the matrices considered are distinct. However, the underlying recursion argument is identical to that used in Proposition \ref{recursion}. \\
		
		Given $u \le r$, let $K_{u, n, r} \subset (\mathbb{P}^n)^{(r)}$ be the locally closed subset of $r$-tuples of \emph{distinct} points of $\mathbb{P}^n$ which form the columns of an $(n + 1) \times r$ matrix of rank $u$. The reasoning in the proof of Proposition \ref{recursion} $[K_{u, n, r}] = [\mathbb{G}(u - 1, n)][K_{u, u - 1, r}]$ in $K_0(\Var_k)$. We fix the linear subspace spanned by the columns of the matrix and consider coordinates of the columns with respect to a fixed basis of this linear subspace. \\

		As in Proposition \ref{recursion}, the definition of $K_{u, n, r}$ implies that
		
		\begin{equation} \label{rec}
			[K_{u, u - 1, r}] = [\UConf_{u - 1} \mathbb{P}^{n - m}] - \sum_{v = 1}^{u - 1} [I_{v, u - 1, r}].
		\end{equation}
		
		where $\UConf_r X \subset X^{(r)}$ denotes unordered $r$-tuples of distinct points on $X$. \\
		
		For each $1 \le v \le u - 1$, the same reasoning as above implies that $[K_{v, u - 1, r}] = [\mathbb{G}(v - 1, u - 1)][K_{v, v - 1, r}]$ and \[ [K_{v, v - 1, r}] = [\UConf_r \mathbb{P}^{v - 1}] - \sum_{w = 1}^{u - 1} [K_{w, v - 1, r}].  \]

		In each step of this recursion, the indices $a, b$ in $K_{a, b, r}$ are strictly smaller than those in the previous step. So, this process must stop after a finite number of steps. Since $[K_{1, b, r}] = 0$ as we're considering distinct points of $\mathbb{P}^{n - m}$ and $[K_{2, b, r}] = [\mathbb{G}(1, b)][K_{2, 1, r}] = [\mathbb{G}(1, b)][\UConf_r \mathbb{P}^1]$, the reduction $[K_{u, n, r}] = [\mathbb{G}(u - 1, n)][K_{u, u - 1, r}]$ followed by induction on $u$ in $K_{u, u - 1, r}$ via the recursion \ref{rec} implies that $[K_{u, n, r}]$ is a polynomial in $\mathbb{L}$ for each $u \le r \le n + 1$ if the unordered configuration spaces $\UConf_r \mathbb{P}^{n - m}$ are polynomials in $\mathbb{L}$. \\
		
		We can show this using the standard decomposition of projective space into affine spaces. A bijection of rational points implies that \[ [\UConf_r X] = \left[ \bigsqcup_{i + j = r} (\UConf_i A \times \UConf_j B) \right] = \sum_{i + j = r} [\UConf_i A] [\UConf_j B] \] if $X = A \sqcup B$ with $A$ and $B$ locally closed in $X$. This reduces the question to showing that $\UConf_r \mathbb{L}^k$ is a polynomial in $\mathbb{L}$, which follows from Lemma \ref{confrec}. \\
		
		Since $[C] = [(\mathbb{P}^{n - m})^{(k + 1)}] - [K_{k + 1, n - m, r}]$ and $[D] = [\UConf_{d - k - 1} \mathbb{P}^{n - m}] - [K_{d - k - 1, n - m, r}]$, they must also be polynomials in $\mathbb{L}$. 
	\end{proof}
	
	Here is the proof of the lemma used in the proof of Proposition \ref{highrecursion}. The main idea is to split to squared and squarefree parts. \\
	
	\begin{lem} \label{confrec}
		Let $K$ be a field of characteristic $0$ and $X$ be an affine variety over $K$. There is a recursive formula for the class of $\UConf^n X$ in $K_0(\Var_K)$: \[ [\UConf^n X] = [\Sym^n X] - \sum_{k \ge 1} [\UConf^{n - 2k} X] [\Sym^k X] \] Note that we use the convention $[\UConf_0 X] = 1$. 
	\end{lem}

	\begin{proof}
		This follows the strategy outlined in the proof of Theorem 1.2 on p. 4 of \cite{FW} (proof on p. 7 -- 9). The main difference is that $\Sym^n X \not\cong X^n$ for arbitrary varieties $X$ if we don't assume $X = \mathbb{A}^n$ or $X = \mathbb{P}^n$. Our assumption that $X$ is affine is used to show that its image under the diagonal map is closed (i.e. $X$ is separated) and that the topology on $X^n/S_n$ is the quotient topology. \\

		Given an element of $X^n$, let $r$ be the number of distinct elements (written $x_1, \ldots, x_r$) and $m_i$ be the multiplicity of $x_i$ (i.e. the number of times $x_i$ appears). Let $Q_k$ be the subset of $X^n$ such that $\sum_{i = 1}^r \left\lfloor \frac{m_i}{2} \right\rfloor \ge k$. This is an analogue of polynomials of degree $n$ such that the squarefree part has degree $\le n - 2k$ (preimage of $m = 1$ and $n = 2$ case of $R_{n, k}^{d, m}$ in p. 7 of \cite{FW}). Note that this is preserved under the action of $S_n$ on $X^n$ which permutes the coordinates. \\
		
		We claim that $Q_k \subset X^n$ is closed. Continuing to put $m = 1$ and $n = 2$ in the proof in \cite{FW}, let $\mathcal{S}$ be the set of injections $\sigma : \{ 1, 2 \} \times \{ 1, \ldots, k \} \hookrightarrow \{ 1, \ldots, n \}$ such that $\sigma(1, a) < \sigma(1, b)$ if $a < b$ and $\sigma(1, j) < \sigma(2, j)$ for each $1 \le j \le k$. The first coordinate corresponds to the ``copy'' of the squared polynomial $h$ in $f = gh^2$ with $g$ squarefree for a particular polynomial $f$. The condition that $\sigma(1, a) < \sigma(1, b)$ means that we only count which $k$-tuples of slots occupied by the roots rather than the particular order that the roots are placed. Similarly, the relative ordering of roots in the first and second copy of $h$ is fixed by the condition $\sigma(1, j) < \sigma(2, j)$. \\
		
		Consider the sets $L_\sigma := \{ x_{\sigma(1, b)} = x_{\sigma(2, b)} \forall 1 \le b \le k \} \subset X^n$. This matches up $k$ of the roots in the two copies of $h$ in some particular collection of slots corresponding to the embedding $\sigma$. For a fixed value of $b$, the points of $X^n$ satisfying the condition are isomorphic to a product of $X^{n - 2}$ with the diagonal $\Delta_X \subset X^2$. Note that $\Delta_X \subset X^2$ is closed since $X$ is affine (and therefore separated). This means that $L_\sigma$ is an intersection of closed subsets of $X^n$, which is closed. The connection to the claim above is that $Q_k = \bigcup_{\sigma \in \mathcal{S}} L_\sigma$. This is because we need to match up at least $k$ pairs to obtain an element of $Q_k$ and the condition defining $L_\sigma$ implies that $\sum_{i = 1}^r \left\lfloor \frac{m_i}{2} \right\rfloor \ge k$ for any element of $\bigcup_{\sigma \in \mathcal{S}} L_\sigma \subset X^n$. Since $Q_k$ is a finite union of closed sets, it is closed. Finally, this implies that $Q_k/S_n$ is closed in $X^n/S_n = \Sym^n X$ since the topology on quotients of affine varieties by finite groups is the quotient topology (e.g. see Proposition 1.1 in \cite{M}). \\
		
		Let $R_k = Q_k/S_n \subset X^n/S_n$. We claim that the map \[ \varphi : \UConf^{n - 2k} X \times \Sym^k X \subset \Sym^{n - 2k} X \times \Sym^k X \longrightarrow R_k \setminus R_{k + 1} \subset \Sym^n X \] induced by the map \[ X^{n - 2k} \times X^k \longrightarrow X^n \] sending $((x_1, \ldots, x_{n - 2k}), (y_1, \ldots, y_k) ) \mapsto ((x_1, \ldots, x_{n - 2k}), (y_1, \ldots, y_k), (y_1, \ldots, y_k)  )$ gives a bijection of rational points in $\overline{k}$. Note that $\varphi$ is \emph{not} an isomorphism even in the case $X = \mathbb{A}^1$ as mentioned on p. 8 of \cite{FW}. It is clear that $\varphi$ is surjective over $k$. The map $\varphi$ is also injective since the (unordered) repeated $k$-tuple and additional $n - 2k$ points (of multiplicity $1$) added uniquely determine an element of $R_k \setminus R_{k + 1}$. Thus, $\varphi$ gives a bijection of $\overline{k}$-rational points. By Proposition 1.4.11 on p. 65 of \cite{CNS}, this implies that $\varphi$ is a piecewise isomorphism and $[\UConf^{n - 2k} X][\Sym^k X] = [R_k] - [R_{k + 1}]$. Since $[\UConf^n X] = [R_0] - [R_1]$ and $R_0 = \Sym^n X$, we can add all the terms to obtain the statement in the proposition.
	\end{proof}

	\begin{exmp} \label{uconfexmp}
		\begin{align*}
			[\UConf^1 X] &= [X] \\
			[\UConf^2 X] &= [\Sym^2 X] - [X] \\
			[\UConf^3 X] &= [\Sym^3 X] - [\Conf^2 X] - [X] \\
			&= [\Sym^3 X] - [X]^2 + [X] - [X] \\
			&= [\Sym^3 X] - [X]^2
		\end{align*}
	\end{exmp}

	All the dimensions here are equal to or bounded above by the dimensions of the analogous degeneracy loci from Section \ref{lowdeg}. Taking $\alpha = \dim \mathbb{G}(n - m, n)$, substituting in these bounds gives the following relative dimensions in $\widehat{\mathcal{K}}$ of terms in the beginning of Section \ref{highdeg}:
	
	\begin{itemize}
		\item Term 1 and 2 from lines \ref{hinit1} and \ref{hinit2}: 
		\begin{itemize}
			\item Main terms $[Y^{(k + 1)}][G(n - m + 1 - (k + 1), n + 1 - (k + 1))]$ and $[J]$: \[ \dim Y^{(k + 1)} + \dim G(n - m + 1 - (k + 1), n + 1 - (k + 1)) - \alpha = m(k + 1) + m(n - m - k) - m(n - m + 1) \\
			= 0 \]
			
			For the second term, we have that \[ \dim J - \alpha = \alpha - \alpha = 0 \] since the projection map $J \longrightarrow \mathbb{G}(n - m, n)$ sending $((p_1, \ldots, p_{d - k - 1}), \Lambda) \mapsto \Lambda$ is surjective and has finite fibers. \\
			
			\item Degenerate terms $[M][G(n - m + 1 - (k + 1), n + 1 - (k + 1))]$ and $[K]$: \[ \dim M + \dim G(n - m + 1 - (k + 1), n + 1 - (k + 1)) - \alpha \le mk - 1 + m(n - m - k) - m(n - m + 1) \\
			= -m - 1  \]

		\end{itemize}
		
		\item Term 3 (tangent planes) from line \ref{hinit2} \begin{align*}
			\dim Q - \alpha &\le S = m + (m - (n - m))(n - m) - m(n - m + 1) \\
			&= m + (m - (n - m))(n - m) - m(n - m) - m \\
			&= -(n - m)^2
		\end{align*}

		\item Term 4 (degenerate incidence correspondences) from line \ref{hinit2}: 
		
		By Proposition \ref{bbound}, Proposition \ref{highmultcont}, and Proposition \ref{smallrem}, we have that \[ \dim \widetilde{B}_2 - \alpha \le -2(n - m - (k - 2) - 1) \] and \[ \dim \widetilde{T}_2 - \alpha = -m(n - m + 1). \]

		\item Term 5 (degeneracies involving $F_{n - m}(Y)$) from line \ref{hinit2}:
		Suppose that $Y$ is contained in a general complete intersection $X$ of hypersurfaces of degrees $d_1, \ldots, d_s$ in $\mathbb{P}^n$, covered by lines, and that a finite number of lines pass through a general point of $Y$. Then, Theorem 2.3 on p. 4 of \cite{CZ} implies that \begin{align*}
			\dim F_{n - m}(Y) + \dim C - \alpha &\le m(n - m + 1) - \sum_{i  = 1}^s \binom{d_i + n - m}{n - m} + (n - m)k + k - 1 - m(n - m + 1) \\
			&= - \sum_{i  = 1}^s \binom{d_i + n - m}{n - m} + (n - m)k + k - 1 
		\end{align*} and
		\begin{align*}
			\dim F_{n - m}(Y) + \dim D - \alpha &\le m(n - m + 1) - \sum_{i  = 1}^s \binom{d_i + n - m}{n - m} + (n - m - 1)(d - k) - (d - k - 1) \\
			&- m(n - m + 1) \\
			&=  - \sum_{i  = 1}^s \binom{d_i + n - m}{n - m} + (n - m - 1)(d - k) - (d - k - 1)
		\end{align*}
		
		Note that $s \le n - m - 1$. We can use bounds on binomial coefficients to study the sizes of the dimensions above.For example, suppose that $d_i \gg n - m$ for each $i$ and use the inquality $\binom{mn}{n} \ge \frac{m^{(n - 1) + 1}}{(m - 1)^{(m - 1)(n - 1)}} n^{-\frac{1}{2}} \ge m^n n^{-\frac{1}{2}}$ for each $m \ge 2$ and $n \ge 1$ (Problem 10819 on p. 652 of \cite{Kr}, p. 2 of \cite{S}) to obtain bounds on suitable variables. Let $r = n - m$. If $d_i = ar$ for each $i$, it suffices to have $k < s \cdot a^r r^{-\frac{3}{2}}$ since $s \cdot a^r r^{-\frac{1}{2}} > r \cdot k \Longleftrightarrow k < s \cdot a^r r^{-\frac{3}{2}}$. This means that the degeneracy term involving $C$ in \ref{hinit2} approaches $0$ in $\widehat{\mathcal{K}}$ as $(n - m) \to \infty$. The relative dimension of the term involving $D$ in \ref{hinit2} may still be large unless $n - m = 1$, in which case we can substitute in $(ar)^s$ for $d$ and obtain suitable bounds.

		\item Variable size restrictions:
		\begin{multicols}{3}
			\begin{itemize}			
				\item $d \ge k + 3$ 
				
				\item $d - k - 1 > n - m - 1$ 
				
				\item $k + 1 \le n - m - 1$ 
				
				\item $n - m \le m - 1$ 
				
				\item $d \ge (n - m) + 2$ 
			\end{itemize}
		\end{multicols}

	\end{itemize}
	
	\section{Limits in $\widehat{\mathcal{K}}$} \label{lims}
	
	The purpose of this section is to combine the dimension computations from Section \ref{lowdeg} and Section \ref{highdeg} (the relative dimension versions at the end of these sections (Definition \ref{reldimdef})) to obtain the limits in Part \ref{endlimit} and Part \ref{highendlim} of Theorem \ref{avglims} in $\widehat{\mathcal{K}}$.

	\subsection{Low degree nondegenerate varieties ($d - k - 1 \le n - m - 1$)} \label{lowlim}
	
	We first apply dimension counts to limits in the low degree setting $d - k - 1 \le n - m - 1$. Afterwards, we substitute the dimensions into the extended $Y-F(Y)$ relation to obtain a limit in $\widehat{\mathcal{K}}$ (Part \ref{endlimit} of Theorem \ref{avglims}). Recall that we have the following relative dimensions (i.e. dimensions in $\widehat{\mathcal{K}}$) and restrictions on variables involved: 
	
	\begin{itemize}
		\item Terms 1 and 2 from lines \ref{init1} and \ref{init2}: 
		\begin{itemize}
			\item Main term: $0$ 
			
			\item Degenerate terms: 
			
			\begin{itemize}
				\item Using $M$: $\le - m - 1$ 
				
				\item Using $N$: $\le - m - 1$ 
			\end{itemize} 
		\end{itemize}
		
		\item Term 3 from line \ref{init2}:

		\begin{itemize}
			\item $\dim P \le - m - 1 + (d - k - 1)$ (Proposition \ref{tandim})

			\item $ \dim Q \le - m - 1 + (k + 1)$ (Proposition \ref{tandim})  
		\end{itemize}
		
		\item Term 4 (degenerate incidence correspondences) from line \ref{init2}: 
		\begin{itemize}
			\item $\dim \widetilde{B}_2 \le - 2(n - m - (k - 2) - 1)$ 
			
			\item $\dim \widetilde{T}_2 \le - 2(n - m - (d - k - 4) - 1)$ 
		\end{itemize}

		\item Term 5 (degeneracies involving $F_{n - m}(Y)$) from line \ref{fano}: 
		
		If $Y$ is contained in some general hypersurface of degree $e$, then \begin{align*}
			\dim F_{n - m}(Y) + \dim C &\le -(n - m)(m - k - 1) + n + k \\
			&<  -(n - m)(m - k - 1) + 2m + k 
		\end{align*} and
		\begin{align*}
			\dim F_{n - m}(Y) + \dim D &\le -(n - m)(m - (d - k - 2) - 1) + n + (d - k - 2) \\
			&< -(n - m)(m - (d - k - 2) - 1) + 2m + (d - k - 2) 
		\end{align*}
		since we assumed that $n - m \le m - 1$. Note that our variable restrictions imply that $m > k + 1$ and $m > d - k - 2$.

		\item Variable size restrictions:
		\begin{multicols}{3}
			\begin{itemize}			
				\item $d \ge k + 3$ 
				
				\item $d - k - 1 \le n - m - 1$ 
				
				\item $k + 1 \le n - m - 1$ 
				
				\item $n - m \le m - 1$ 
				
				\item $d \ge (n - m) + 2$ 
			\end{itemize}
		\end{multicols}

	\end{itemize}
	
	In order for the dimensions of the degenerate loci in Terms 1, 2, and 4 (from lines \ref{init1}, \ref{init2}) to approach $-\infty$, it suffices to have $m \to \infty$ and $n - m \to \infty$ as $k \to \infty$ ``reasonably quickly''. Note that the dimension of Term 3 approaches $-\infty$ since $n - m$ is much larger than $d - k - 1$, $k + 1$, or $m - 1$ under the assumptions of Part 1 of Theorem \ref{avglims}. For Term 5, it suffices to take $m - k \to \infty$ and $m - (d - k - 2) \to \infty$ as $k \to \infty$ if we assume that $(n - m) - k \to \infty$. Note that these are consistent with our variable restrictions since substituting in the fourth restriction to the second and third ones imply that $d - k - 1 \le m - 2$ and $k + 1 \le m - 2$. Putting the ranges above together gives the limit from Part \ref{endlimit} of Theorem \ref{avglims} in $\widehat{\mathcal{K}}$. 
	
	\begin{rem}
		In the example values, we chose $d = (n - m) + \lfloor \sqrt{k} \rfloor$ to ensure that $Y$ is a nondegenerate variety since $d \ge 2 + (n - m)$ if $Y$ is nondegenerate and not a rational normal scroll or Veronese surface. Note that the Veronese surfaces do not affect what happens in the limit. The main purpose is to find parameters which may apply to a more varied collection of varieties. It is clear that the sample values given satisfy the variable restrictions above. 
	\end{rem}
	
	We end with further details on Example \ref{lowexmp} from the introduction.
		
		\begin{exmp}(Low degree examples for Part \ref{endlimit} of Theorem \ref{avglims} from Example \ref{lowexmp}: Linear subspaces contained in scrolls and (hyper)quadric fibrations)  \label{lowexmpext} \\
			There is a classification of smooth $m$-dimensional varieties $Y \subset \mathbb{P}^n$ of degree $d \le 2(n - m) + 1$ not contained in a hyperplane (i.e. nondegenerate). We will only consider varieties where the dimension $m$ can be arbitrarily large and contain $(n - m)$-planes (Theorem I on p. 339 of \cite{I2}). Two of the three families of such varieties (excluding quadric hypersurfaces and $\mathbb{P}^n$) which can take an arbitrarily large dimension with $d \le 2(n - m) + 1$ are scrolls over curves or surfaces and (hyper)quadric fibrations. When the base of these scrolls and (hyper)quadric fibrations is $\mathbb{P}^1$, we can make some concrete observations on the Fano varieties of $k$-planes on these varieties. \\
			
			In the case of scrolls over $\mathbb{P}^1$, we have a complete description (Proposition 2.2 on p. 4066 of \cite{L}). The $k$-planes contained in such a scroll are either contained in a $(k - 1)$-plane inside the fiber of the defining projection map or the span of lines involved in the construction of the scroll as the span of a collection of rational normal curves. Theorem 1.5 on p. 511 of \cite{LPS} gives a fiberwise embedding of any (hyper)quadric fibration $X \longrightarrow \mathbb{P}^1$ (paired with a very ample line bundle $\mathcal{L}$) over $\mathbb{P}^1$ into a projective bundle $\mathbb{P}(\mathcal{E})$ over $\mathbb{P}^1$ which connects these (hyper)quadric fibrations to scrolls over $\mathbb{P}^1$. \\
			
			Recall that a scroll $\mathbb{P}(\mathcal{E})$ over $\mathbb{P}^1$ can also be defined as the image of a vector bundle over $\mathbb{P}^1$ with a particular embedding $\mathcal{O}_{\mathbb{P}(\mathcal{E})}(1)$ into projective space (p. 5 of \cite{EH2}). We also have that the restriction of the tautological line bundle on $\mathbb{P}(\mathcal{E})$ to $X$ is equal to $\mathcal{L}$ (equation 1.0.3 on p. 509 of \cite{LPS}). These two embeddings can be combined to study linear subspaces of $X$ via their images in the scroll over $\mathbb{P}^1$. This can likely be translated into a concrete problem since scrolls over $\mathbb{P}^1$ can also be defined as the vanishing locus of $2 \times 2$ minors of a certain matrix (Exercise 9.11 on p. 106 of \cite{H1}). \\
		\end{exmp}

	\subsection{Higher degree varieties ($d - k - 1 > n - m - 1$)} \label{highlim}
	
	The same reasoning as Section \ref{lowlim} can be used to obtain the limit from Part \ref{highendlim} of Theorem \ref{avglims} in $\widehat{\mathcal{K}}$. Also, the limit in Part \ref{highendlim} of Theorem \ref{avglims} has a particularly simple expression when we apply the point 
	counting motivic measure and assume some divisibility conditions. 
	
	\begin{defn} (Definition 4.3.4 on p. 112 of \cite{CNS}) \label{sep} \\
		A motivic measure $\mu$ with values in a ring $A$ is \emph{separated} if there is a morphism of rings $\overline{\mu} : \overline{\mathcal{M}_k} \longrightarrow A$ such that $\mu(X) = \overline{\mu}([X])$ for each $k$-variety $X$. This is equivalent to a motivic measure $\mu: K_0(\Var_k) \longrightarrow A$ satisfying the following conditions: 
		\begin{itemize}
			\item $\mu(\mathbb{L}) \in A^\times$ 
			
			\item $\widetilde{\mu}(F^\infty \mathcal{M}_k) = 0$, where $\widetilde{\mu} : \mathcal{M}_k \longrightarrow A$ is the unique ring homomorphism such that $\widetilde{\mu}([X] \mathbb{L}^{-i}) = \mu(X) \mu(\mathbb{L})^{-i}$. 
		\end{itemize}
		
	\end{defn}
	
	\begin{prop}
		The extended point counting motivic measure $\mu : K_0(\Var_k) \longrightarrow \mathbb{Q}$ sending $[X] \mapsto \# X(\mathbb{F}_q)$ is separated. 
	\end{prop}

	Now that we have some idea of how motivic measures of $K_0(\Var_K)$ interact with the completion $\widehat{\mathcal{K}}$, we will give a proof of Corollary \ref{counting}. 
	
	\begin{proof}(Proof of Corollary \ref{counting}) \\
		We will write $\#$ in place of the notation $\#_{q, e}$ from Corollary \ref{counting} for the $\mathbb{F}_{q^e}$-point count. Recall from Section \ref{highdeg} that \[ J = \{ ((p_1, \ldots, p_{d - k - 1}), \Lambda) \in W : \text{$p_i$ distinct, } \dim \overline{p_1, \ldots, p_{d - k - 1}} = n - m \}. \]  First consider the subset $\widetilde{J} \subset J$ coming from points of $\mathbb{G}(n - m, n) \setminus B$, where $B = \{ \Lambda \in \mathbb{G}(n - m, n) : \Lambda \text{ tangent to } Y \}$. By Proposition \ref{tandim}, we have that \[ \dim B \le m + (m - (n - m))(n - m). \] This means that $ \frac{\# B(\mathbb{F}_{q^e})}{\# \mathbb{G}(n - m, n) (\mathbb{F}_{q^e})} = O(q^{-e((n - m)^2})$. Our assumption that $e > \binom{d}{d - k - 1}$ implies that $\binom{d}{d - k - 1} \frac{\# B(\mathbb{F}_{q^e})}{\# \mathbb{G}(n - m, n) (\mathbb{F}_{q^e})} \to 0$ in the limit (which takes $(n - m) \to \infty$). This also takes care of elements of $J \setminus \widetilde{J}$ where $\Lambda \not\subset Y$. Before moving to elements of $J \setminus \widetilde{J}$, we will continue to obtain point counts for elements of $J$ with $\Lambda \not\subset Y$. \\ 
		
		The preimage of each point of $\mathbb{G}(n - m, n) \setminus B$ is a collection of $d$ points over $\overline{\mathbb{F}_q}$. What we would like to find are $(d - k - 1)$-tuples of points on $Y$ lying on a given $(n - m)$-plane that are invariant under the action of $\Gal(\overline{\mathbb{F}_q}/\mathbb{F}_q)$. Our assumption on $\mathbb{F}_q$-irreducible components of $(d - k - 1)$-tuples lying on $Y \cap \Lambda$ implies that there are no $\mathbb{F}_{q^e}$-points coming from $J$ by the following modification of the Lang--Weil bound applied to $\mathbb{F}_q$-irreducible components of $Y \cap \Lambda$ for each $\Lambda \in \mathbb{G}(n - m, n) \setminus B$.

		\begin{prop}(Modified Lang--Weil bound, Proposition 3.1 on p. 6 of \cite{M2}) \label{mlw} \\
			Suppose that $K = \mathbb{F}_q$ is a finite field and let $X \subset \mathbb{P}_k^n$ be an irreducible closed subvariety of degree $d$ and dimension $r$. Denote by $\Gamma = \{ W_1, \ldots, W_m \}$ the set of irreducible components of $X_{\overline{k}} = X \times_k \overline{k}$. 
			
			There are positive constants $c_X$ and $c'_X$ such that for every $e \ge 1$, we have 
			
			\[ \begin{cases} 
				|\# X(\mathbb{F}_{q^e}) - mq^{er}| \le \frac{(d - m)(d - 2m)}{m} q^{e\left( r - \frac{1}{2} \right)} + c_X q^{e(r - 1)} & \text{ if } m | e \text{ and } \\
				\# X(\mathbb{F}_{q^e}) \le c'_X q^{e(r - 1)} & \text{ if } m \nmid e.
			\end{cases}
			\]
			
			Furthermore, if $X$ is smooth over $\mathbb{F}_q$, then we may take $c'_X = 0$ and $c_X$ to only depend on $n, d,$ and $r$ (but not on $X$ or on $k$).
		\end{prop} 
		
		If $N | e$ in Corollary \ref{counting}, then we use the first part of Proposition \ref{mlw}. Note that the degree of a finite set as a variety is its cardinality. Alternatively, a simpler method for finding $\mathbb{F}_q$-points of $(d - k - 1)$-tuples of $Y$ lying on $Y \cap \Lambda$ for some fixed $\Lambda \in \mathbb{G}(n - m, n) \setminus B$ is to count collections of $\Gal(\overline{\mathbb{F}_q}/\mathbb{F}_q)$-orbits of points on d $Y \cap \Lambda$ which have cardinality adding to $d - k - 1$. \\ 
		
		In the $0$-dimensional case, the number of $\mathbb{F}_q$-irreducible components (denoted $m$ in Proposition \ref{mlw}) is the number of $\mathbb{F}_q$-points. This applies to our setting since $\pi^{-1}(\Lambda)$ is finite for all $\Lambda \in \mathbb{G}(n - m, n) \setminus B$ even over the algebraic closure. Since $1 \le \# \pi^{-1}(\Lambda)(\mathbb{F}_q) \le \binom{d}{d - k - 1}$ for each $\Lambda \in \mathbb{G}(n - m, n) \setminus B$, we have that $m | e$ for any $\Lambda$ if $e$ is divisible by $\binom{d}{d - k - 1}!$. After base changing $\pi^{-1}(\Lambda)_{\mathbb{F}_q}$ to $\mathbb{F}_{q^e}$, we end up with the same number of points as in the algebraic closure $\overline{\mathbb{F}_q}$ (see proof of Proposition 3.1 on p. 6 of \cite{M2}). Putting together the ``irreducible components'' (which are really just $\mathbb{F}_q$-points), we find that the number of $\mathbb{F}_{q^e}$-points in $\pi^{-1}(\Lambda)$ is $\binom{d}{d - k - 1}$ for each $\Lambda \in \mathbb{G}(n - m, n) \setminus B$. In general, we  add the number of geometric points in each $\Gal(\overline{\mathbb{F}_q}/\mathbb{F}_q)$-orbit whose size divides $e$ to get the total number of $\mathbb{F}_{q^e}$-points in $\pi^{-1}(\Lambda)$. \\

		We finally consider point counts of elements of $J$ such that $\Lambda \subset Y$. These come from $\Lambda \in F_{n - m}(Y)$ and the $p_i$ represent $(d - k - 1)$-tuples on an $(n - m)$-plane $\Lambda \in F_{n - m}(Y)$ which span the entire plane $\Lambda$. To do this, we compare point counts of linearly independent $(k + 1)$-tuples and $(d - k - 1)$-tuples lying on an $(n - m)$-plane to find an approximation. \\ 
		
		Since $\widetilde{T}_1 = \emptyset$ by Proposition \ref{highmultcont}, the $(d - k - 1)$-tuples in question must span an $(n - m)$-plane. We will also omit $(n - m)$-planes which contain linearly dependent $(k + 1)$-tuples since terms associated to them approach $0$ in the completion. Note that every element of $\widetilde{A}$ (linearly independent $(k + 1)$-tuples in $Y$ contained in an element of $F_{n - m}(Y)$) comes a $(d - k - 1)$-tuple in $Y$ spanning an $(n - m)$-plane of $\widetilde{J} \setminus J$. For each element of $\widetilde{R}$, we have $\binom{d - k - 1}{k + 1}$ possible choices of $(k + 1)$-tuples. However, we also need to take into account redundancies by determining the space of $(d - k - 1)$-tuples in $Y$ spanning an $(n - m)$-plane in $F_{n - m}(Y)$ which contain a given linearly independent $(k + 1)$-tuple in $Y$ paired with an $(n - m)$-plane in $F_{n - m}(Y)$ containing it. This is the space of $(d - 2k - 2)$-tuples in $\mathbb{P}^{n - m} \setminus \{ \text{$k$ points} \}$ spanning an $(n - m)$-plane. \\ 
		
		Given the point count for $\widetilde{A}$, this gives an approximate count \[ \# \widetilde{A}(\mathbb{F}_{q^e}) \approx \frac{\binom{d - k - 1}{k + 1} \# \widetilde{R}(\mathbb{F}_{q^e}) }{ \# \{ ((p_1, \ldots, p_{d - 2k - 2}), \Lambda) : \text{ $p_i$ distinct in $\mathbb{P}^{n - m}$, } \dim \overline{p_1, \ldots, p_{d - 2k - 2}} = n - m \} (\mathbb{F}_{q^e}) }, \]
		 
		where the denominator parametrizes $(d - k - 1)$-tuples containing a fixed $(k + 1)$-tuple. This denominator can be approximated by $\UConf_{d - 2k - 2} (\mathbb{P}^{n - m} \setminus \{ \text{$k + 1$ points} \}$. \\
		
		From the point of view of $\# \widetilde{R}(\mathbb{F}_{q^e})$, this means that  \[ \# \widetilde{R}(\mathbb{F}_{q^e}) \approx \frac{\# \{ ((p_1, \ldots, p_{d - 2k - 2}), \Lambda) : \text{ $p_i$ distinct in $\mathbb{P}^{n - m}$, } \dim \overline{p_1, \ldots, p_{d - 2k - 2}} = n - m \} (\mathbb{F}_{q^e})  }{ \binom{d - k - 1}{k + 1} } \# \widetilde{A}(\mathbb{F}_{q^e}). \] 
		
		An upper bound would be given by $\# \UConf_{d - 2k - 2} (\mathbb{P}^{n - m} \setminus \{ \text{$k + 1$ points} \}) (\mathbb{F}_{q^e})$. This was computed recursively in Lemma \ref{confrec}. For a more precise estimate, one could attempt to follow the steps of Proposition \ref{highrecursion} to compute the class of distinct points of $\mathbb{P}^{n - m} \setminus \{ \text{$k + 1$ points} \}$ spanning an $(n - m)$-plane. \\

		We can consider the ratio between point counts of $V \setminus \widetilde{\widetilde{R}}$ and $W \setminus \widetilde{\widetilde{A}}$, where $\widetilde{\widetilde{R}} \subset V$ is the subset of $(d - k - 1)$-tuples of distinct points lying in an $(n - m)$-plane contained in $Y$ paired with this plane and $\widetilde{\widetilde{A}} \subset W$ denotes $(k + 1)$-tuples of distinct points lying in an $(n - m)$-plane contained in $Y$. Applying the proof of Proposition \ref{extend}, we obtain a correspondence between elements of $W \setminus \widetilde{\widetilde{A}}$ and those of $V \setminus \widetilde{\widetilde{R}}$. \\

		The two ratios of point counts can be used to compare the point count of $V$ with that of $W$. Writing $\# X := X(\mathbb{F}_{q^e})$, the dimension counts in the proof of Theorem \ref{avglims} imply that
		
		\begin{align*}
			\frac{\# V}{\# W} &= \frac{\# V - \# \widetilde{\widetilde{R}}}{\# W} + \frac{\# \widetilde{\widetilde{R}}}{\# W} \\
			&= \frac{\# V - \# \widetilde{ \widetilde{R}}}{\# W - \# \widetilde{\widetilde{A}}} \cdot \frac{\# W - \widetilde{\widetilde{A}}}{\# W} + \frac{\widetilde{\widetilde{R}}}{\# \widetilde{\widetilde{A}}} \cdot \frac{\# \widetilde{\widetilde{A}}}{\# W} \\
			&=  1 \cdot \left( 1 - \frac{\# \widetilde{\widetilde{A}}}{\# W} \right) + \frac{\widetilde{\widetilde{R}}}{\# \widetilde{\widetilde{A}}} \cdot \frac{\# \widetilde{\widetilde{A}}}{\# W} \\
			&= \left( 1 -  \Theta(q^{e((k + 1)(n - m) + \dim F_{n - m}(Y) - m(n - m + 1))}) \right) + \Theta(q^{e(n - m)(d - 2k - 2)}) \cdot  \Theta(q^{e((k + 1)(n - m) + \dim F_{n - m}(Y) - m(n - m + 1))}) \\
			&= 1 + \Theta(q^{e(n - m)(d - 2k - 2)})(\Theta(q^{e((k + 1)(n - m) + \dim F_{n - m}(Y) - m(n - m + 1))}) - 1) \\
			&= \Theta(q^{e((n - m)(d - k - 1) + \dim F_{n - m}(Y) - m(n - m + 1))}).
		\end{align*}
		 
		Note that $\# \widetilde{\widetilde{A}}$ and $\# \widetilde{\widetilde{R}}$ are polynomials in $q$ which can be written in terms of multinomial coefficients (see Remark \ref{recursionrem}, proofs of Proposition \ref{recursion} and Proposition \ref{highrecursion}). Recall that $d - k - 1 \gg n$ in Part \ref{highendlim} of Theorem \ref{avglims}. This means that $\frac{\# V}{\# W}$ is very large. Since the limit is taken as $n - m \to \infty$ and all the other variables are taken to be functions of the codimension $n - m$, the term is dominated by large powers of $q$. 
		Let $u = \frac{\# V}{\# W}$. Substituting $\# V = u \# W$ into in $[W] - [\widetilde{B}] - [\widetilde{A}] = [V] - [\widetilde{R}] - [\widetilde{T}]$ in Proposition \ref{extend}, the point counts are
		 
		\begin{align*}
			\# W - \# \widetilde{A} - \# \widetilde{B} &= \# V - \# \widetilde{R} - \# T \\
			 \# W - \# \widetilde{A} - \# \widetilde{B} &=  u \# W - \# \widetilde{R} - \# T \\
			- \# \widetilde{R} - \# \widetilde{T} &= (1 - u) \# W - \# \widetilde{A} - \# \widetilde{B}.
		\end{align*} 
	
		At this point, we can either group the terms $W$, $\widetilde{A}$, and $\widetilde{B}_1$ together to use the point count $J \setminus \widetilde{J}$ above or estimate point counts of individual terms on the right hand side. The first method gives the following decomposition.
		
		\begin{align*}
			- \# \widetilde{R} - \# \widetilde{T} &= (1 - u) \# W - \# \widetilde{A} - \# \widetilde{B} \\
			&= (1 - u)(\# W - \# \widetilde{A} - \# \widetilde{B}_1) - u \# \widetilde{A} - u \# \widetilde{B}_1 - \# \widetilde{B}_2 \\
			\Rightarrow - \frac{\# \widetilde{R}}{\# \mathbb{G}(n - m)} - \frac{\# \widetilde{T}_1}{\# \mathbb{G}(n - m, n)} - \frac{\# \widetilde{T}_2}{\# \mathbb{G}(n - m, n)} &= \frac{(1 - u) (\# W - \# \widetilde{A} - \# \widetilde{B}_1)}{\# \mathbb{G}(n - m, n)} - \frac{u \# \widetilde{A}}{\# \mathbb{G}(n - m, n)} \\ &- \frac{u \# \widetilde{B}_1}{\# \mathbb{G}(n - m, n)} - \frac{\# \widetilde{B}_2}{\# \mathbb{G}(n - m, n)} 
		\end{align*}
	
		The terms $\# \widetilde{R}$ and $\widetilde{A}$ can be expressed in terms of $\# F_{n - m}(Y)$ and polynomials in $q$ since $[\widetilde{R}]$ and $[\widetilde{A}]$ are $[F_{n - m}(Y)]$ multiplied by a product in $\mathbb{L}$. Also, note that the limiting estimate $\frac{\# W - \# \widetilde{A} - \# \widetilde{B}_1}{\# \mathbb{G}(n - m, n)}$ is given by that of $\frac{\#  J \setminus \widetilde{J}}{\# \mathbb{G}(n - m, n)}$ above. Since $\frac{\# \widetilde{T}_2}{\# \mathbb{G}(n - m, n)}$ and $\frac{\# \widetilde{B}_2}{\# \mathbb{G}(n - m, n)} $ vanish in the limit as $(n - m) \to \infty$, it remains to find estimates for $\# \widetilde{T}_1$ and $\# \widetilde{B}_1$ (which parametrize incidence correspondences of linearly dependent $(k + 1)$-tuples and $(d - k - 1)$-tuples with linear span of dimension $\le n - m - 1$ respectively). Note that terms which involve $(n - m)$-planes $\Lambda \subset Y$ can be absorbed into $\# \widetilde{R}$ and $\# \widetilde{A}$ to get $\# \widetilde{\widetilde{R}}$ and $\# \widetilde{\widetilde{A}}$. Then, the proof of Proposition \ref{highmultcont} implies that there are no terms of $\widetilde{B}_1$ with $\Lambda \not\subset Y$ and $\# \widetilde{T}_1 = \Theta(q^{km - (n - m - k + 1)})$ by Proposition \ref{lam1} with $\lambda = k - 2$. \\
		
		Substituting in the dimension estimates along with the subset $D$ of $(d - k - 1)$-tuples in $\mathbb{P}^{n - m}$ spanning a linear subspace of dimension $\le n - m - 1$ and the subset $C$ of linearly independent $(k + 1)$-tuples (whose classes in $K_0(\Var_K)$ a polynomial in $\mathbb{L}$ by Proposition \ref{recursion} and Proposition \ref{highrecursion}), we find that 
		
		\begin{align*}
			- \frac{\# \widetilde{R}}{\# \mathbb{G}(n - m)} - \frac{\# \widetilde{T}_1}{\# \mathbb{G}(n - m, n)} - \frac{\# \widetilde{T}_2}{\# \mathbb{G}(n - m, n)} &= \frac{(1 - u) (\# W - \# \widetilde{A} - \# \widetilde{B}_1)}{\# \mathbb{G}(n - m, n)} - \frac{u \# \widetilde{A}}{\# \mathbb{G}(n - m, n)} \\ &- \frac{u \# \widetilde{B}_1}{\# \mathbb{G}(n - m, n)} - \frac{\# \widetilde{B}_2}{\# \mathbb{G}(n - m, n)} 
		\end{align*}
	
		and 
		
		\begin{align*}
			 - \frac{\#_{q, e} F_{n - m}(Y) (\#_{q, e} \UConf_{d - k - 1}(\mathbb{P}^{n - m}) - \#_{q, e} D)}{\#_{q, e} \mathbb{G}(n - m, n)} - \Theta(q^{km - (n - m - k + 1) - m(n - m + 1)}) - \frac{\#_{q, e} \widetilde{T}_2}{\#_{q, e} \mathbb{G}(n - m, n)} \\ 
			 = (1 - u) \alpha - \frac{u \#_{q, e} F_{n - m}(Y) ((\#_{q, e} \mathbb{P}^{n - m})^{(k + 1)} - \#_{q, e} C)}{\#_{q, e} \mathbb{G}(n - m, n)} - \frac{\#_{q, e} \widetilde{B}_2}{\#_{q, e} \mathbb{G}(n - m, n)},
		\end{align*}
		 
		where $0 \le \alpha \le \binom{d}{d - k - 1}$ with $\alpha = 0$ if $N \nmid e$ for each $N \in T_\Lambda$ from $\Lambda \in \mathbb{G}(n - m, n)$ such that $|Y \cap \Lambda| = d$ (see Corollary \ref{counting}) and $\alpha = \binom{d}{k + 1}$ if $e$ is divisible by $\binom{d}{d - k - 1}!$ (Proposition \ref{mlw}). Note that $u = 1 - \beta + \beta f$, where $\beta = \Theta(q^{e((k + 1)(n - m) + \dim F_{n - m}(Y) - m(n - m + 1))})$ and $f$ is a rational function in ${q^e}$ determined by $\frac{[\widetilde{\widetilde{R}}]}{[\widetilde{\widetilde{A}}]}$, which is a rational function in $\mathbb{L}$ of degree $(k + 1)(n - m) + \dim F_{n - m}(Y) - m(n - m + 1)$. Given a fixed $q$, this implies that \[ \lim_{n - m \to \infty} \frac{\#_{q, e} F_{n - m}(Y) (\#_{q, e} \UConf_{d - k - 1}(\mathbb{P}^{n - m}) - \#_{q, e} D - u\#_{q, e} (\mathbb{P}^{n - m})^{(k + 1)} + \#_{q, e} C)}{\#_{q, e} \mathbb{G}(n - m, n)} + (1 - u)\alpha + \gamma = 0  \] in the limit for some $\gamma = \Theta(q^{e(km - (n - m - k + 1) - m(n - m + 1))})$ that varies with the initial parameters, which are all functions of $n - m$.

	\end{proof}
	
	\begin{rem}\label{countcom}   ~\\
		\vspace{-7mm}
		
		\begin{enumerate}
			
			\item This reasoning with divisibility conditions does \emph{not} imply that $\# Y^{(k + 1)}(\mathbb{F}_{q^e}) = 0$ since the divisibility conditions in Corollary \ref{counting} from Proposition \ref{mlw} are used on $\mathbb{F}_q$-irreducible components of $(d - k - 1)$-tuples, \emph{not} $(k + 1)$-tuples.

			\item The divisibility condition in Corollary \ref{counting} from Proposition \ref{mlw} is least restrictive when $e$ is prime. Equality holds exactly when the size of each $\mathbb{F}_q$-irreducible component is $1$ for each $\Lambda \in \mathbb{G}(n - m, n)$ such that $|Y \cap \Lambda| = d$. For example, this would occur if the map $J \longrightarrow \mathbb{G}(n - m, n)$ is a piecewise trivial fibration above its image instead of just being a covering map. Note that this requires checking images over non-closed points of $\mathbb{G}(n - m, n)$. 
			
			\item In Theorem \ref{avglims}, the connection between variations of $\mathbb{F}_q$-point counts with the covering map $J \longrightarrow \mathbb{G}(n - m, n)$ is that the monodromy action induced  could be involved in studying how the point count above is distributed among conjugacy classes of the action of Frobenius on general $(n - m)$-plane sections of $Y$ (see \cite{Ent}). 
			
			\item A natural question to ask is what the distribution of point counts of $J$ behave if we impose additional geometric restrictions on the type of variety $Y$ while varying the codimension, dimension and degree. 
		
			\item Part \ref{highendlim} of Theorem \ref{avglims} can be applied to any separated motivic measure in place of finite field point counts. For example, the following result implies that we can use the \'etale representation for the Euler characteristic (interpreted as graded respresentations): 
			
			\begin{prop}(Corollary 4.3.9 on p. 113 of \cite{CNS}) \label{compat2} \\
				The \'etale motivic measure $\chi_{\text{\'et}} : K_0(\Var_k)[\mathbb{L}^{-1}] \longrightarrow K_0(\Rep_{G_k} \mathbb{Q}_l)$ given by \[ \chi_{\text{\'et}}(X) = \sum_{n \ge 0} (- 1)^n [H^n_{\text{\'et}}(X \otimes_k k^s, \mathbb{Q}_l)] \] for separated $k$-varieties $X$ (p. 95 of \cite{CNS}) is a separated motivic measure. 
			\end{prop}

		\end{enumerate}
	\end{rem}

	Here is an example of varieties where the point count above and Part 2 of Theorem \ref{avglims} applies.
	
	\begin{exmp}(High degree examples for Part \ref{highendlim} of Theorem \ref{avglims} from Example \ref{highexmp}: Complete intersections of generic hypersurfaces of large degree) \label{highexmpext} \\
		We can analyze the relative sizes of the variables to show that there are many terms where the relative dimensions of the main term involving the Fano $(n - m)$-planes do \emph{not} vanish in the completion unlike the linear dependence non-generic terms. Suppose that $Y \subset \mathbb{P}^n$ is general a complete intersection. Since $\dim Y = m$, this means that $Y$ is a complete intersection of $n - m$ hypersurfaces. By Theorem 2.4 on p. 4 of \cite{CZ}, we have that \[ \dim F_{n - m}(Y) = m(n - m + 1) - \sum_{i = 1}^{n - m} \binom{d_i + n - m}{n - m}. \]

		The relative dimension (Definition \ref{reldimdef}) of \[ \frac{[F_{n - m}(Y)]([(\mathbb{P}^{n - m})^{(k + 1)}] - [(\mathbb{P}^{n - m})^{(d - k - 1)}])}{[\mathbb{G}(n - m, n)]} \] is then \[ (d - k - 1)(n - m) - \sum_{i = 1}^{n - m} \binom{d_i + n - m}{n - m}. \]
		
		If the degrees $d_i$ are sufficiently large, then the first term involving Fano $(n - m)$-planes Part \ref{highendlim} of Theorem \ref{avglims} (a multiple of the term above) does \emph{not} vanish in the limit. Taking the point sample size $k \ll n - m$ means that complete intersections of generic hypersurfaces of large degree (relative to $n - m$) which are $u$-linearly generic for $u \le d - k - 2$ give examples where Part \ref{highendlim} of Theorem \ref{avglims} applies.  
		
	\end{exmp}

	Department of Mathematics, University of Chicago \\
	5734 S. University Ave, Office: E 128 \\ Chicago, IL 60637 \\
	\textcolor{white}{text} \\
	Email address: \href{mailto:shpg@uchicago.edu}{shpg@uchicago.edu}

\end{document}